\documentclass[11pt]{article}
\usepackage{amsmath,amsfonts,amssymb,mathrsfs,cases,mathdots,colordvi,url}
\usepackage{enumerate}
\usepackage{amssymb,latexsym}
\usepackage{graphicx}
\usepackage{amstext}
\usepackage{amsmath}
\usepackage{amsopn}
\usepackage{amsbsy}
\usepackage{amscd}
\usepackage{amsxtra}
\usepackage{graphics}
\usepackage{booktabs,multirow}
\usepackage{tabularx}
\usepackage{makecell}
\usepackage{bbding}
\usepackage{pifont}
\usepackage{hyperref}
\usepackage{tikz}
\usepackage{pifont}
\usepackage{amsthm}
\usepackage{graphicx}
\usepackage{subfigure}
\usepackage{hyperref}

\usetikzlibrary{shapes.geometric,arrows}
\usetikzlibrary{positioning, shapes.geometric}
\tikzstyle{block} = [rectangle, draw, fill=white!50,
text width=33em, text centered, rounded corners, minimum height=3em]
\tikzstyle{line} = [draw, -latex']
\tikzstyle{arrow} = [thick,->,>=stealth]
\DeclareMathOperator*{\argmin}{arg\,min}
\DeclareMathOperator*{\argmax}{arg\,max}
\numberwithin{equation}{section}

\newtheorem{theorem}{Theorem}[section]
\newtheorem{definition}{Definition}[section]
\newtheorem{proposition}{Proposition}[section]

\newtheorem{lemma}{Lemma}[section]
\newtheorem{corollary}{Corollary}[section]

\begin{document}
	
	\title{Calm local optimality for couple-constrained minimax problems 
		\thanks{The alphabetical order of the paper indicates the equal contribution to the paper.}
		}
	\author{Xiaoxiao Ma\thanks{National Supercomputing Center in Shenzhen, Shenzhen, Guangdong 518055, China. Email: xxma.research@outlook.com.}, Jane Ye\thanks{Corresponding author. Department of Mathematics and Statistics, University of Victoria, Victoria, B.C., Canada V8W 2Y2.  The research of this author was partially supported by NSERC. Email: janeye@uvic.ca.}  }
	\date{}
	\maketitle
	\begin{abstract}
	Recently, a new local optimality concept for minimax problems, termed calm local minimax points, has been introduced.
	In this paper, we extend this concept to a general class of nonsmooth, nonconvex–nonconcave minimax problems with coupled constraints, where the inner feasible set depends on the outer variable. We derive comprehensive first-order and second-order necessary and sufficient optimality conditions for calm local minimax points in the setting of nonsmooth, nonconvex–nonconcave minimax problems with coupled constraints. Furthermore, we show how these conditions apply to problems with set constraints, as well as those involving systems of inequalities and equalities. By unifying existing formulations that often rely on stronger assumptions within the framework of calm local minimax points, we show that our results hold under weaker assumptions than those previously required.

		\vskip 10 true pt
		\noindent {\bf Key words.}\quad Minimax problem, Coupled constraints, Local optimality, Calmness, First-order optimality condition, Second-order optimality condition
		\vskip 10 true pt
		
		
	\end{abstract}
	\newpage

\section{Introduction}
We consider the following minimax problem:
\begin{equation}\label{minimax}\tag{Min-Max}
	\min_{x\in X} \max_{y\in Y(x)} f(x,y),
\end{equation}
where the objective function $f:\mathbb{R}^n\times\mathbb{R}^m\rightarrow\mathbb{R}$ is possibly nonsmooth, and the sets $X \subseteq\mathbb{R}^n$, $Y(x)\subseteq\mathbb{R}^m$ (for each $x \in X$) are nonempty, closed, but may be nonconvex. Throughout the paper, we assume that for each $x \in X$, the solution set of the inner maximization problem $\max _{y^{\prime} \in Y(x)} f(x, y^{\prime})$ is nonempty.

The minimax problem with  coupled constraints \eqref{minimax} exhibits marked differences from the simple minimax problem where the feasible set {of $y$} is independent of variable $x$:
As shown in \cite{Tsaknakis2021} for the case of linear coupled constraints, such problems may violate the classical max–min inequality and can be NP-hard, even when the objective is strongly convex–strongly concave and $Y(x)$ contains only a linear inequality constraint. Minimax problems with  coupled-constraints arise in various applications, including adversarial training \cite{Aleksander18,Tsaknakis2021} and generative adversarial networks \cite{Gauthier2019,Ian2014}. For example, \cite{LuMei2023,Tsaknakis2021} illustrate how such problems can be applied to model resource allocation and network flow scenarios in the presence of adversarial attacks. Further applications, along with algorithmic and theoretical developments, are discussed in \cite{DaiWa2022,DaiZh-2020,DaiZh-2022,Goktas2022,Goktas2023,GoktasZhao2022,LuMei2023,Tsaknakis2021,ZhangWang2022}.

A point $(\bar x,\bar y) \in X\times Y(\bar x)$  is said to be a global minimax point  of problem (Min-Max) if for any $x \in X, y\in  Y(\bar x)$,
$$f(\bar x, y) \leq f(\bar x,\bar y) \leq \max _{y^{\prime} \in Y(x)} f(x, y^{\prime}).$$
However in optimization,  a global optimal solution is hard to find and in practice one usually try to find stationary points or local optimal solutions as local surrogates for global optimal solutions. 
To address this difficulty, Jin et al. \cite{JNJ20}  introduced the notion of local minimax points for unconstrained minimax problems.  The notion was extended to the  couple-constrained case by Dai and Zhang \cite{DaiZh-2020} as follows.  A
  point $(\bar x,\bar y) \in X\times Y(\bar x)$  is said to be a local minimax point  of problem (Min-Max) if there exists a $\delta_0>0 $ and a radius function $\tau:\mathbb{R}_+\rightarrow \mathbb{R}_+$ satisfying $\tau (\delta) \rightarrow 0$ as $\delta \downarrow 0$, such that for any $\delta\in (0,\delta_0]$ and any $x\in X \cap \mathbb{B}_\delta(\bar x)$, $y\in Y( \bar x) \cap \mathbb{B}_\delta(\bar y)$ we have
$$f(\bar x, y) \leq f(\bar x,\bar y) \leq \max _{y^{\prime} \in Y(x)\cap \mathbb{B}_{\tau(\delta)}(\bar y)} f(x, y^{\prime}).$$Subsequent works \cite{DaiZh-2020, JC22, JNJ20, Zhang2022} have further explored properties and optimality conditions for local minimax problems, and \cite{Xu2023} studied related bilevel problems.

More recently, Ma et al. \cite{Ma2023}  demonstrated that, to explicitly characterize local minimax behavior, it is essential to study a special subclass termed calm local minimax points which is a special local minimax point where   the radius function $\tau$  is calm, i.e., there is $\kappa>0 $ such that $\tau(\delta)\leq \kappa \delta$,
 and presented a detailed analysis of first- and second-order optimality conditions for the   simple minimax problem.


In this paper, we show that calm local minimax points play a central role in analyzing local optimality of minimax problems, as many existing results in fact implicitly rely on this subclass rather than on general local minimax points. This highlights the necessity of a deeper study of calm local minimax points. However, existing research on (calm) local minimax points either assumes smooth data or restricts attention to the simple minimax problem. Extending the analysis to the couple-constrained case \eqref{minimax} is nontrivial, due to its intrinsic differences. 

In this paper we derive second-order optimality conditions for problem (Min-Max) only under the assumption of {twice  semidifferentiability} for the objective function $f$ and the semidifferentiability {and the calmness} of the set-valued map $Y(x)$. 

In practice, constraint sets $X$ and $Y(x)$ are usually defined by a sysem of smooth equality and inequality constraints. For this special but important case we derive our optimality conditions in terms of Lagrange functions.  To illustrate our results in the smooth equality and inequality constrainted case, for simplicity, consider the special case where
\begin{eqnarray*}
X:= \{ x\in \mathbb{R}^n | \phi(x)\leq 0\},\quad
Y(x):= \{ y\in \mathbb{R}^m | \varphi (x,y)\leq 0\},
\end{eqnarray*}
where $f:\mathbb{R}^n\times\mathbb{R}^m\rightarrow\mathbb{R}$, $\phi:\mathbb{R}^n\rightarrow\mathbb{R}^p $ and $\varphi:\mathbb{R}^n\times\mathbb{R}^m\rightarrow\mathbb{R}^q$ are twice continuously differentiable. Denote the Lagrangian function of the minimax problem by
\begin{eqnarray*}
L(x,y,\alpha,\beta) := f(x,y) + \phi(x)^T\alpha - \varphi(x,y)^T\beta.
\end{eqnarray*}
 Let $(\bar x,\bar y)$ be a calm local minimax point and suppose  that Mangasarian-Fromovitz constraint qualification (MFCQ) holds for the system $\varphi(\bar x, y)\leq 0$ at $\bar y$ and metric subregulariy constraint qualification (MSCQ) holds for the sytem $\phi(x)\leq 0$ at $\bar x$, then there exists multipliers $\alpha, \beta$ such that the following dual first-order optimality  condition holds:
\begin{equation}\label{multipliersetminimax} \nabla_{(x,y)} L(\bar x,\bar y,\alpha,\beta)=0, 0\leq -\phi(\bar x) \perp \alpha \geq 0,  0\leq -\varphi ( \bar x, \bar y) \perp \beta \geq 0.\end{equation}
To our knowledge,  there is no dual first-order necessary optimality condition existed in such generality. In \cite[Theorem 4.1]{Ma2023}, only the first-order optimality conditions in primary form for the simple minimax problem was given. Note that in the case where linear independence constraint qualification (LICQ) instead of MFCQ holds for the sysem $\varphi(\bar x, y)\leq 0$ at $\bar y$, the multipliers $\beta $ are unique and the condition is the same as the one derived in 
 \cite[Theorems 3.1]{DaiZh-2020} under the Jacobian uniqueness condition which essentially implies that the solution map of the inner maximization problem has a single-valued  twice continuously differentiable localization.
 To state the  second-order optimality condition we derived, define
  the critical cone for the outer minimization problem as  $$ C_{\min}(\bar x,\bar y)
	=\{ u|\nabla \phi_i(\bar x)^T u \leq 0, i\in I_\phi(\bar x), \sup\limits_{h \in \mathbb{L}(\bar x,\bar y;u)} \nabla f(\bar x,\bar y)^T(u,h) =  0\},$$
	where  $\mathbb{L}(\bar x,\bar y;u) :=\{h'|\nabla \varphi_i(\bar x,\bar y)^T (u,h') \leq 0, i \in I_{\varphi}(\bar x,\bar y)\}.$
 Then for any critical direction $u \in C_{\min}(\bar x,\bar y)$, there exist $h$ 
  in the set  $$ C(\bar x,\bar y;u)= 
	 \{h'| \nabla \varphi_i(\bar x,\bar y)^T (u,h') \leq 0, i \in I_{\varphi}(\bar x,\bar y), \nabla f(\bar x,\bar y)^T (u,h') = 0 \}
$$
 and  $(\alpha,\beta)$ which is a multiplier  of the minimax problem satisfying (\ref{multipliersetminimax}) such that the second-order necessary optimality condition holds
\begin{equation}\label{withhnew}
\nabla^2_{(x,y)} L(\bar x,\bar y,\alpha,\beta)((u,h),(u,h))  \geq 0.
\end{equation}
We have also shown that under a strong second-order optimality condition for the inner maximization problem, the above necessary optimality condition becomes a suffficent one when the inequality in  (\ref{withhnew}) is changed to a strict inequality.
Under some extra conditions we also give an exact and explicit form of the left-hand side of \eqref{withhnew}, thereby yielding both second-order necessary and sufficient optimality conditions with explicitly defined $h$. This improves upon \cite[Corollary 5.1]{Ma2023}, where only an upper bound of the left-hand side of \eqref{withhnew} is obtained by relaxing the feasibility of $h$, and thus no explicit sufficient condition with explicitly defined $h$  was  derived.

	The main contributions of this paper are as follows:
	\begin{itemize}
		\item We develop a comprehensive framework of first- and second-order optimality conditions for calm local minimax points in the couple-constrained setting~\eqref{minimax} under very weak and general assumptions. This extends and even improves existing results, which have largely been limited to smooth data or the simple constraints form. Moreover, we provide a constructive characterization of the existence direction $h$ in the second-order optimality conditions, thereby offering a clear and explicit representation rather than a purely existential statement. This explicit form reveals the intrinsic structure underlying the optimality conditions.
		
		\item We develop a refined sensitivity analysis of value functions under general constraint systems. In particular, we derive a detailed first-order sensitivity analysis of the localized value function, overcoming technical difficulties caused by nonsmoothness and localization constraints, and establish its role as a foundation for deriving necessary and sufficient optimality conditions.
		
		\item We unify several existing approaches to optimality conditions for local minimax points within the framework of calm local minimax points. Our analysis demonstrates that many formulations—often derived under stronger assumptions such as the Jacobian uniqueness condition—can be interpreted as special cases of this framework. This highlights calm local minimax points as the natural and essential concept for a complete theory of local optimality in minimax problems.
\end{itemize}

The remainder of this paper is organized as follows. In Section~\ref{4.2}, we present preliminaries from variational analysis along with some preliminary results. Section~\ref{section-points} introduces the notion of calm local minimax points for couple-constrained minimax problems and gives several equivalent characterizations. In Section~\ref{nonsmooth-section}, we establish first- and second-order necessary and sufficient optimality conditions for calm local minimax points under coupled constraints. Section~\ref{setconstrained} extends the analysis to problems with set constraints, as well as inequalities and equalities systems, and compares the results with existing ones. Conclusions are given in Section \ref{concl}.

\section{Notations and Preliminaries}\label{4.2}
In this section, we will introduce variational analysis tools that are essential for our analysis of local optimality in the minimax problem.

{\bf Notations:} We denote by $\mathbb{R}_+^r (\mathbb{R}_-^r)$ the nonnegative (nonpositive) orthant. For any $z \in \mathbb{R}^r$, $\|z\|$ denotes its Euclidean norm. For $z \in \mathbb{R}^r$ and $\epsilon>0$, we denote by $\mathbb{B}_\epsilon(z):=\{z' \,| \,\|z'-z\|\leq \epsilon\}$ the closed ball centered at $z$ with radius $\epsilon$ and by $\mathbb{B}$ the closed unit ball. For any two vectors $a, b$ in $\mathbb{R}^r$, we denote by $\langle a, b\rangle$ the inner product. For any $z \in \mathbb{R}^r$ and $S \subseteq \mathbb{R}^r$, $\mathrm{dist}(z,S):=\inf_{z'\in S}\|z-z'\|$. For a set $S \subseteq \mathbb{R}^r$,  the indicator function is defined by {$\delta_S( z)= 0 $  { if }  $z \in S$ and $\delta_S( z)= \infty $ \text { otherwise},  and $S^{\perp}:=\{\alpha \in \mathbb{R}^r | \langle \alpha, z\rangle=0, \forall z \in S\}$ denotes the orthogonal complement. For a set  $S\subseteq \mathbb{R}^r$, a point $\bar z \in \mathbb{R}^r$, and a sequence $z_k$, the notation $z_k \stackrel{S}{\rightarrow} \bar{z}$ means that the sequence $z_k\in S$ goes to $\bar z$. The notation $l(t) = o(t)$ means $l(t)/t \to 0$ as $ t \downarrow 0$. For a set-valued mapping $\Gamma: \mathbb{R}^n \rightrightarrows \mathbb{R}^m$, ${\rm gph\ }\Gamma:=\{(x,y) \in \mathbb{R}^n \times \mathbb{R}^m | x \in \mathbb{R}^n, y\in \Gamma(x)\}$ denotes the graph of $\Gamma$.  For a single-valued map $\Phi: \mathbb{R}^r \rightarrow \mathbb{R}$, we denote by $\nabla \Phi(z) \in \mathbb{R}^{r}$ and $\nabla^2 \Phi(z) \in \mathbb{R}^{r \times r}$ the gradient vector of $\Phi$ at $z$ and the Hessian matrix of $\Phi$ at $z$, respectively. {If $ f=(f_1,\dots, f_m): \mathbb{R}^n \rightarrow \mathbb{R}^m$ is a vector function that is twice differentiable at $\bar z \in \mathbb{R}^n$,  we denote by $\nabla f(\bar z) \in \mathbb{R}^{m\times n}$ its Jacobian and $\nabla^2 f(\bar z)$ its second derivative. Throughout the paper, the notation $\nabla^2 f(\bar z)(w,v)$ means that
		$$\nabla^2 f(\bar z)(w,v):=(w^T\nabla^2 f_1(\bar z)v, \dots, w^T \nabla^2 f_m(\bar z)v) \mbox{ for all } v, w \in \mathbb{R}^n. $$  } For a matrix $A \in \mathbb{R}^{n \times m}, A^T$ is its transpose and rank$\{A\}$ denotes its rank. For a symmetric matrix $A \in \mathbb{R}^{r \times r},$ $A\prec 0\ $ means that the matrix $A$ is a negative definite matrix and $A^{-1}$ is the inverse matrix.

Consider  a set-valued map $\Gamma: \mathbb{R}^n \rightrightarrows \mathbb{R}^m$. The Painlevé-Kuratowski outer limit and inner limit of $\Gamma$ with respect to a set $S$ at $\bar z$ is defined by
$$
\limsup _{z \overset{S}{\to} \bar z} \Gamma(z):=\left\{v \in \mathbb{R}^m \mid \exists  z_k \overset{S}{\to} \bar z, v_k \rightarrow v \text { s.t. } v_k \in \Gamma\left(z_k\right) \text { for each } k\right\},
$$
$$
\liminf _{z \overset{S}{\to} \bar z} \Gamma(z):=\left\{v \in \mathbb{R}^m \mid \forall z_k \overset{S}{\to} \bar z, \exists v_k \rightarrow v \text { s.t. } v_k \in \Gamma\left(z_k\right) \text { for each } k\right\},
$$
respectively.

\begin{definition}[tangent and normal cones \cite{BonSh00,RoWe98}]
Given  $S\subseteq\mathbb{R}^r, \bar z\in S$, the tangent/contingent cone and the inner tangent cone to $S$ at $\bar z$ are defined by 
\begin{eqnarray*}
 T_S(\bar z)&:=&
  \big\{w \in\mathbb{R}^r \, \big| \, \exists \ t_k\downarrow 0,\; w_k\to w \ \ {\rm with}
 \ \ \bar z+t_k w_k\in S \big\},\\
 T^i_S(\bar z)&:=&
  \big\{w \in\mathbb{R}^r \, \big| \, \forall \ t_k\downarrow 0,\; \exists w_k\to w \ \ {\rm with}
 \ \ \bar z+t_k w_k\in S \big\},
 \end{eqnarray*}
respectively.
 The regular/Fr\'echet normal cone 
 and the limiting/Mordukhovich normal cone to $S$ at $\bar z$ are given, respectively, by
 \begin{eqnarray*}
 \widehat{N}_S(\bar z)&:=&\left\{z^*\in \mathbb{R}^r \, \big| \, \langle z^*, z-\bar z\rangle \le
 o\big(\|z-\bar z\|\big) \ \forall z\in S\right\},\\
 N_S(\bar z)&:=&
 \left \{z^* \in \mathbb{R}^r \, \Big| \, \exists \ z_k \overset{S}{\to} \bar z, z^*_{k}\rightarrow z^* \ {\rm with }\  z^*_{k}\in \widehat{N}_{S}(z_k) \right \}.
\end{eqnarray*}
\end{definition}

The regular normal cone to $S$ at $\bar{z}$ \cite[Proposition 6.5]{RoWe98} can also be characterized by
\begin{equation}\label{normalcone-equ}
\widehat{N}_{S}(\bar{z}):=\left\{z^* \in \mathbb{R}^{r} \mid \langle z^*, w\rangle \leq 0 \quad \forall w\in T_S(\bar z)\right\} = T_S(\bar z)^{\circ}.
\end{equation}
For a closed set $S$, one always has $\widehat{N}_{S}(\bar{z}) \subseteq {N}_{S}(\bar{z})$, where the two cones agree and reduce to the normal cone of convex analysis if $S$ is convex. 
A set $S \subseteq \mathbb{R}^r$ is said to be geometrically derivable at $\bar z \in S$ if 
$T_S(\bar z) = \lim_{t \downarrow 0} \frac{S - \bar z}{t},$
or equivalently, if $T_S(\bar z) = T_S^i(\bar z).$
 Convex sets are geometrically derivable.

It is well-known that in the convex case, the normal cone and the tangent cone are polar to each other.

\begin{proposition} [Tangent-Normal Polarity] \cite[Theorem 6.28]{RoWe98}\label{prop_polar} For a closed convex set $S$ and $\bar{z} \in S$, one has
$$
T_{S}(\bar{z})=N_{S}(\bar{z})^{\circ}, \quad T_{S}(\bar{z})^{\circ}=N_{S}(\bar{z}) .
$$
\end{proposition}

In particular, the set $T_{S}(\bar{z})$ is closed and convex when $S$ is closed.

\begin{definition}[graphical derivatives]\cite[Definition 8.33]{RoWe98}\label{def2.2} Consider a mapping $Y: \mathbb{R}^n \rightrightarrows \mathbb{R}^m$ and a point $\bar{x} \in \operatorname{dom} Y$. The graphical derivative of $Y$ at $\bar{x}$ for any $\bar{y} \in Y(\bar{x})$ is the mapping $D Y(\bar{x},\bar{y}): \mathbb{R}^n \rightrightarrows \mathbb{R}^m$ defined by
$$
h \in D Y(\bar{x},\bar{y})(u) \quad\Longleftrightarrow \quad (u, h) \in T_{\operatorname{gph} Y}(\bar{x}, \bar{y}).
$$
\end{definition}

Using the Painlevé-Kuratowski outer limit, the graphical derivative defined in Definition~\ref{def2.2} can be expressed as
\begin{equation}\label{Graphical}
D Y(\bar{x},\bar{y})(u) = \limsup_{t \downarrow 0, u' \to u} \frac{Y(\bar x+tu')-\bar y}{t} \quad \forall u\in\mathbb{R}^n.
\end{equation}

\begin{definition}[semidifferentiability of set-valued mappings]\cite[page 332]{RoWe98}\label{semid-set}
Consider  a set-valued map $Y: \mathbb{R}^n \rightrightarrows \mathbb{R}^m$ and $(\bar x,\bar y) \in  \operatorname{gph}Y$. The limit
$$\lim_{t \downarrow 0, u' \to u} \frac{Y(\bar x+tu')-\bar y}{t},$$
if it exists, is the semiderivative at $\bar x$ for $\bar y$ and $u$. If it exists for every vector $u \in \mathbb{R}^n$, then $Y$ is semidifferentiable at $\bar x$ for $\bar y$.
\end{definition}

\begin{definition}[Lipschitz-like property and calmness of set-valued mappings]\label{lip-calm}
Consider  a set-valued map $Y: \mathbb{R}^n \rightrightarrows \mathbb{R}^m$ and $(\bar x,\bar y) \in  \operatorname{gph}Y$. We say that $Y$ is Lipschitz-like or pseudo Lipschitz continuous or satisfies Aubin property around $(\bar x,\bar y)$ \cite[Definition 1]{Aubin84},  \cite{DoRo2009}, if
there exists a constant $l>0$ and neighborhoods $U$ of $\bar x$ and $V$ of  $\bar y$ such that
$$Y(x) \cap V \subseteq Y(x')+l\|x-x'\|\mathbb{B} \quad \forall x,x'\in U.$$

We say that $Y$ is calm (or pseudo upper-Lipschitz continuous) around $(\bar x,\bar y)$ \cite[Definition 2.8]{YeYe1997}, if
there exists a constant $l>0$ and neighborhoods $U$ of $\bar x$ and $V$ of  $\bar y$ such that
$$Y(x) \cap V \subseteq \bar y+l\|x-\bar x\|\mathbb{B} \quad \forall x\in U.$$
\end{definition}

It is straightforward to verify that the Lipschitz-like property implies the calmness property for a set-valued mapping.

\begin{definition}[metric subregularity constraint qualification]
Consider the constraint system
\begin{equation}\label{constraints} S=\{z \in \mathbb{R}^r|g(z)\in \Sigma\},
\end{equation}
where $g:\mathbb{R}^r \rightarrow \mathbb{R}^{q}$ and  $\Sigma \subseteq \mathbb{R}^q$  is closed. Let $\bar z\in S$ where $S$ is the constraint system defined by (\ref{constraints}). We say that the metric subregularity constraint qualification (MSCQ)  for  $S$ holds at  $\bar z$ if there exists a neighborhood $U$ of $\bar z$ and a constant $\rho>0$ such that
$${\rm dist}(z,S) \leq \rho\ {\rm dist}(g(z),\Sigma) \qquad \forall z\in U.$$
\end{definition}
Sufficient conditions for MSCQ of the inequalities and equalities system can be found in \cite[Theorem 7.4]{YeZhou18}, e.g., the first-order sufficient condition for metric subregularity (FOSCMS), the second-order sufficient condition for metric subregularity (SOSCMS), the Mangasarian-Fromovitz constraint qualification (MFCQ), and the linear constraint qualification, i.e., $g$ is affine and $\Sigma$ is the union of finitely many polyhedral convex sets.

The following discussions are important for deriving the first-order sufficient optimality conditions. 
\begin{proposition}\label{prop2.1}
Consider  a set-valued map $Y: \mathbb{R}^n \rightrightarrows \mathbb{R}^m$ and $(\bar x,\bar y) \in  \operatorname{gph}Y$. Then (i) $\Leftrightarrow$ (ii). If we further assume that $Y$ is calm around $(\bar x,\bar y)$, then (ii) $\Rightarrow$ (iii) (and thus (i) $\Rightarrow$ (iii)).
\begin{itemize}
\item[(i)] $Y$ is semidifferentiable at $\bar x$ for $\bar y$.
\item[(ii)] For any $u \in \mathbb{R}^n$,
$$D Y(\bar x,\bar y)(u) =  \liminf_{t\downarrow 0,u' \to u} \frac{Y(\bar x+tu')-\bar y}{t} \neq \emptyset.$$
\item[(iii)] There exists $\kappa>0$ such that for any $u \in \mathbb{R}^n$, $h \in D Y(\bar x,\bar y)(u)$, $t_k \downarrow 0, u_k \to u$, there exists a sequence $y_k \in Y(\bar x+t_ku_k)$ such that $(y_k-\bar y)/t_k \to h$ as $k \to \infty$ and $\|y_k-\bar y\| \leq \kappa \|x_k-\bar x\|$ with $x_k:=\bar x+t_ku_k$ for sufficiently large $k$.
\end{itemize}
\end{proposition}
\begin{proof}
By the definition of the Painlevé-Kuratowski inner limit, together with \eqref{Graphical} and Definition~\ref{semid-set}, the statement that $Y$ is semidifferentiable at $\bar{x}$ for $\bar{y}$ is equivalent to condition~(ii).
We next prove that condition~(ii) combined with the calmness of $Y$ yields condition~(iii). Condition (ii) states that for any $u \in \mathbb{R}^n$, $h \in D Y(\bar x,\bar y)(u)$, $t_k \downarrow 0, u_k \to u$, there exists a sequence $h_k \to h$ such that $y_k:=\bar y+t_kh_k \in Y(\bar x+t_ku_k)$ for any $k$. Clearly, $y_k \to \bar y$ as $k$ goes to infinity. By the definition of the calmness, see Definition \ref{lip-calm}, there exists $\kappa>0$ such that $\|y_k-\bar y\| \leq \kappa \|x_k-\bar x\|$ for sufficiently large $k$.
\end{proof}

\begin{definition}[Robinson stability] \cite[Definition 1.1]{Gfrerer2017}
Let $Y(x)$ be a set-valued map defined by $Y(x):= \{ y\in \mathbb{R}^m | \varphi (x,y) \in D\}$ where $\varphi:\mathbb{R}^n \times \mathbb{R}^m  \to \mathbb{R}^q$ and $D\subseteq \mathbb{R}^{q}$ is closed. We say $Y$ satisfies the Robinson stability (RS) property at $(\bar x,\bar y) \in \operatorname{gph}{Y}$ with modulus $\kappa \geq 0$ if there are neighborhoods $U$ of $\bar x$ and $V$ of  $\bar y$ such that
$$\operatorname{dist}(y; Y(x)) \leq \kappa \operatorname{dist}(\varphi(x,y);D) \quad \forall (x,y)\in U\times V.$$
\end{definition}
RS property is also called R-regulariy in e.g. \cite[Definition 2.6]{Bednarczuk2020}.
It is obvious that $Y$ satisfies the Robinson stability (RS) property at $(\bar x,\bar y)$ means that for each fixed $ x$ in   a neighborhood of $\bar x$, MSCQ holds for the system $Y(x)$ at $\bar y$ and the modulus is uniform  for all $x$ in the neighborhood.
\begin{proposition}[sufficient conditions for the Robinson stability]\label{prop2.3}
Let $Y(x):= \{ y\in \mathbb{R}^m | \varphi (x,y) \in D\}$ where $D$ is closed and $\varphi$ be differentiable in $y$ and with continuous partial derivative $\nabla_y \varphi$. Given $(\bar x,\bar y) \in \operatorname{gph}{Y}$. If one of the following properties holds,
\begin{itemize}
\item[(i)] $D=\mathbb{R}^{q_1}_- \times \{0\}^{q_2}$ and MFCQ holds for system $\varphi(\bar x,y )\in D$   at   $\bar y$, i.e., the vectors $\nabla_y \varphi_i(\bar x,\bar y), i=1,...,q_2$ are linearly independent and also there exists $w \in \mathbb{R}^m$ such that $\nabla_y \varphi_i(\bar x,\bar y)w=0,i=1,...,q_2$ and $\nabla_y \varphi_i(\bar x,\bar y)w<0, i\in I(\bar x,\bar y):=\{i=1,..,q_1 | \varphi_i(\bar x,\bar y)=0\}$, 
\item[(ii)] $\varphi(x, y)=a(x) + By + c$, where $a: \mathbb{R}^n \to \mathbb{R}^q$ is continuous, $B \in \mathbb{R}^{q \times m}$ and $c \in \mathbb{R}^p$, $D$ is the union of finitely many convex polyhedral sets, and the feasible region $Y(x)$ is nonempty near $\bar x$,
{\item[(iii)]  $D=\mathbb{R}^{q_1}_- \times \{0\}^{q_2}$ and $Y$ satisfies the relaxed constant rank constraint qualification (RCRCQ) at $(\bar x,\bar y)$, i.e.,  for any index set $K \subseteq I(\bar x,\bar y)$,
$$rank\{\nabla_y \varphi_i(x,y): i \in I_0 \cup K\} = rank \{\nabla_y \varphi_i (\bar x,\bar y): i \in I_0 \cup K\}$$
for all $(x,y)$ in a neighbourhood of $(\bar x,\bar y)$,}
\end{itemize}  where $I(\bar x,\bar y):=\{i=1,..,q_1 | \varphi_i(\bar x,\bar y)=0\}, I_0 = \{j | j = 1,2,...,q_2\}.$
Then $Y$ satisfies the RS property at $(\bar x,\bar y)$. 
\end{proposition}
\begin{proof}
Statement~(i) follows from \cite[Corollary 3.7]{Gfrerer2017} and the equivalence between the MFCQ and metric regularity for systems of inequalities and equalities systems; see, e.g., \cite[Theorem 4.1]{Dontchev2006}. Statement~(ii) is taken from \cite[Proposition 3.2]{BaiYe2023}. Statement~(iii) follows from \cite[Theorem 4.2]{Bednarczuk2020}.
\end{proof}

More sufficient conditions for the RS property can be found in \cite{Bednarczuk2020,Mehlitz2022,Minchenko2011}.

Next, we show that the Robinson stability is a sufficient condition for the semidifferentiability of a set-valued mapping.

\begin{proposition}\label{prop2.2}
Let $Y(x):= \{ y\in \mathbb{R}^m | \varphi (x,y) \in D\}$, $\varphi$ be differentiable, and $D$ be geometrically derivable. Suppose that $Y$ satisfies the Robinson stability (RS) property at $(\bar x,\bar y) \in \operatorname{gph}{Y}$ with modulus $\kappa \geq 0$.
 Then, $Y$ is semidifferentiable at $\bar x$ for $\bar y$ and Lipschitz-like around $(\bar x,\bar y)$.
\end{proposition}
\begin{proof}
Since $\operatorname{dist}((x,y); \operatorname{gph}Y) \leq \operatorname{dist}(y; Y(x))$, the RS property implies that MSCQ holds for the system $\varphi(x,y)\in D$ at $(\bar x,\bar y)$. By \cite[Proposition 4.2]{ABM21},
$$\mathbb{L}_{\operatorname{gph} Y}(\bar x,\bar y;u) = D Y(\bar x,\bar y)(u),$$
where
$$\mathbb{L}_{\operatorname{gph} Y}(\bar x,\bar y;u):=\{h|\nabla \varphi(\bar x,\bar y)(u,h) \in T_D(\varphi(\bar x,\bar y))\}.$$

By Proposition \ref{prop2.1} (ii) and the fact \eqref{Graphical}, we need to show that $$D Y(\bar x,\bar y)(u) \subseteq \liminf_{t\downarrow 0,u' \to u} \frac{Y(\bar x+tu')-\bar y}{t}.$$ Under the assumed RS property, $Y$ is Lipschitz-like around $(\bar x,\bar y)$ \cite[Theorem 5.1]{Gfrerer2017}, and thus, $\liminf_{t\downarrow 0,u' \to u} \frac{Y(\bar x+tu')-\bar y}{t}= \liminf_{t\downarrow 0} \frac{Y(\bar x+tu)-\bar y}{t} $. Thus, we only need to prove that $$\mathbb{L}_{\operatorname{gph} Y}(\bar x,\bar y;u) \subseteq \liminf_{t\downarrow 0} \frac{Y(\bar x+tu)-\bar y}{t}.$$

Let $h \in \mathbb{L}_{\operatorname{gph} Y}(\bar x,\bar y;u)$. Then, by the definition of $\mathbb{L}_{\operatorname{gph} Y}(\bar x,\bar y;u)$ and the assumed geometric derivability, we have
$$
\nabla \varphi(\bar x,\bar y)(u,h) \in T_D(\varphi(\bar x,\bar y)) = T^i_D(\varphi(\bar x,\bar y)).
$$

By the definition of inner tangent sets, we have
$$
\operatorname{dist}\left(\varphi(\bar x,\bar y)+t \nabla \varphi(\bar x,\bar y)(u,h), D\right)=o(t) \quad \forall t \geq 0 .
$$

Thus, it follows from the Robinson stability and the above equality that
$$
\begin{aligned}
& \quad \operatorname{dist}\left(\bar y+th, Y\left(\bar x+tu\right)\right) \\
& \leq \kappa \operatorname{dist}\left(\varphi\left(\bar x+tu, \bar y+th\right), D\right) \\
& =\kappa \operatorname{dist}\left(\varphi(\bar x,\bar y)+t \nabla \varphi(\bar x,\bar y)(u,h) +o(t), D\right) \\
& =\kappa \operatorname{dist}\left(\varphi(\bar x,\bar y)+t \nabla \varphi(\bar x,\bar y)(u,h), D\right) +o(t) \\
& =o(t),
\end{aligned}
$$
where $\kappa \geq 0$, which implies that $h \in \liminf_{t\downarrow 0} \frac{Y(\bar x+tu)-\bar y}{t}$. Thus, we have
$$
\mathbb{L}_{\operatorname{gph} Y}(\bar x,\bar y;u) \subseteq \liminf_{t\downarrow 0} \frac{Y(\bar x+tu)-\bar y}{t}.
$$
\end{proof}

\begin{definition} \text{\cite[Definitions 8.1 and 7.20]{RoWe98}}]
Consider a function $\psi: \mathbb{R}^r \rightarrow \overline{\mathbb{R}}$, a point $\bar z$ with $\psi(\bar z)$ finite, and $w \in  \mathbb{R}^r$.  The subderivative and the superderivative of  $\psi$ at $\bar z$ for  $w$ is defined by
\begin{eqnarray*} {\rm d}\psi(\bar z)(w)&:=& \liminf\limits_{{t\downarrow 0} \atop {w'\to w}} \frac{\psi(\bar{z}+tw')-\psi(\bar{z})}{t},\\
{\rm d}^+\psi (\bar z)(w) &:=& \limsup\limits_{{t\downarrow 0} \atop {w'\to w}} \frac{\psi(\bar z+tw')-\psi(\bar z)}{t},
\end{eqnarray*}
respectively.
When  the limit
$${\rm d}\psi(\bar z)(w)={\rm d}^+\psi (\bar z)(w)=\lim_{{t\downarrow 0} \atop {w'\to w}} \frac{\psi(\bar z+tw')-\psi(\bar z)}{t}$$  exists, we say $\psi$ is semidifferentiable at $\bar z$ for $w$ (or Hadamard differentiable at $\bar z$ in direction $w$). Further if $\psi$ is semidifferentiable at $\bar z$ for every $w$, we say that $\psi$ is semidifferentiable at $\bar z$. It is easy to see that if $\psi$ is semidifferentiable at $\bar z$ for $w$, then
\begin{equation}\label{eqn2.2new}{\rm d}(-\psi)(\bar z)(w)=-{\rm d}\psi(\bar z)(w).\end{equation}
 \end{definition}

According to \cite[Theorem 7.21]{RoWe98}, if the function $\psi$ is semidifferentiable at the point $\bar{z}$, then the subderivative ${\rm d}\psi(\bar{z})(w)$ is finite for every $w \in \mathbb{R}^r$. Moreover, $\psi$ is continuous at $\bar{z}$, and the mapping $w \mapsto {\rm d}\psi(\bar{z})(w)$ is both positively homogeneous and continuous. Furthermore, when $\psi$ is semidifferentiable at $\bar{z}$ for direction $w$, it follows that $\psi$ is directionally differentiable at $\bar{z}$ in $w$, and its subderivative and superderivative coincide with the classical directional derivative of $\psi$ at $\bar{z}$ in the direction $w$. That is,
\begin{equation*}{\rm d}\psi(\bar z)(w)={\rm d}^+\psi (\bar z)(w)=\psi'(\bar z; w):=\lim _{t \downarrow 0} \frac{\psi(\bar z+t w)-\psi(\bar z)}{t}.\label{eqn2.2}\end{equation*}
Furthermore, if $\psi$ is Lipschitz continuous around $\bar z$, then directional differentiability and semidifferentiability of $\psi$ at $\bar{z}$ are equivalent. In addition, when $\psi$ is continuously differentiable at $\bar{z}$, it is necessarily semidifferentiable there, and for any direction $w$, the subderivative satisfies  $ {\rm d}\psi(\bar{z})(w) = \psi'(\bar{z}; w) = \nabla \psi(\bar{z})^{T} w,$ as stated in \cite[Exercise 8.20]{RoWe98}.

\begin{definition}\label{def2.6}  \text{\cite[Defintions 13.3 and 13.6]{RoWe98}}] Let $\psi: \mathbb{R}^r \rightarrow \overline{\mathbb{R}}$, $\psi(\bar z)$ be finite and $\bar v, w\in  \mathbb{R}^r$. The second subderivative  of $\psi$ at $\bar z$ for $\bar v$  and $w$ is
 \begin{eqnarray*}
 {\rm d}^2\psi(\bar z;\bar v)(w):= \liminf\limits_{{t\downarrow 0} \atop {w'\to w}} \frac{\psi(\bar{z}+tw')-\psi(\bar{z})-t\langle \bar{v}, w'\rangle}{\frac{1}{2}t^2}.
\end{eqnarray*}
On the other hand, the second subderivative  of $\psi$ at $\bar z$ for $w$ (without mention of $\bar v$) is defined by
 \begin{eqnarray}\label{secondsub}
 {\rm d}^2\psi(\bar z)(w):= \liminf\limits_{{t\downarrow 0} \atop {w'\to w}} \frac{\psi(\bar{z}+tw')-\psi(\bar{z})-t {\rm d} \psi(\bar z)(w')}{\frac{1}{2}t^2},
\end{eqnarray}
where the sum of $\infty$ and $-\infty$ is interpreted as $\infty$.
The function $\psi$ is twice semidifferentiable at $\bar z$ if it is semidifferentiable at $\bar z$ and  the ``liminf'' in  (\ref{secondsub}) is replaced by the ``lim''
 for any $w \in  \mathbb{R}^r$.
 \end{definition}

If $\psi$ is twice semidifferentiable at $\bar{z}$ in direction $w$, then $\psi$ is also twice directionally differentiable at $\bar{z}$ in the same direction, and its second subderivative coincides with the classical second directional derivative at $\bar{z}$ in direction $w$. That is,
 \begin{equation*}{\rm d}^2\psi(\bar z)(w)=\psi''(\bar z; w):=\lim _{t \downarrow 0} \frac{\psi(\bar z+t w)-\psi(\bar z) -t \psi'( \bar z;w)}{\frac{1}{2} t^2}. \label{eqn2.5new}\end{equation*} 
It follows directly that if $\psi$ is twice semidifferentiable at $\bar{z}$, then
 \begin{equation}\label{twice-epi}
{\rm d}^2 (-\psi)(\bar z)(w)=-{\rm d}^2 \psi(\bar z)(w).
 \end{equation}

According to \cite[Exercise 13.7]{RoWe98}, if $\psi$ is twice semidifferentiable at $\bar{z}$, then the second subderivative ${\rm d}^2 \psi(\bar{z})(w)$ is finite for all $w \in \mathbb{R}^r$. Moreover, when $\psi$ is twice continuously differentiable at $\bar{z}$, it is also twice semidifferentiable there, and for $\bar{v} = \nabla \psi(\bar{z})$, the equality
${\rm d}^2 \psi(\bar{z}; \bar{v})(w) = {\rm d}^2 \psi(\bar{z})(w) = w^{T} \nabla^2 \psi(\bar{z}) w$ holds, as illustrated in \cite[Example 13.8]{RoWe98}.

 By Definition \ref{def2.6}, for a set $S \subseteq \mathbb{R}^r, \bar z\in S,$ $\bar v \in \mathbb{R}^r$, and $w \in T_S(\bar z)$, the second subderivative of the indicator function $\delta_S$ at $\bar z$ for
 $\bar v$ and $w$ is
  \begin{equation*}\label{indicator}
\begin{aligned}
 {\rm d}^2\delta_S(\bar z;\bar v)(w):= \liminf\limits_{{t\downarrow 0} \atop {w'\to w}} \frac{\delta_S(\bar{z}+tw')-t\langle \bar{v}, w'\rangle}{\frac{1}{2}t^2}
= \liminf\limits_{{t\downarrow 0,w'\to w} \atop {\bar{z}+tw' \in S}} \frac{-2\langle \bar{v}, w'\rangle}{t}.
\end{aligned}
 \end{equation*}

The following definition is needed to address nonsmooth functions $\psi$.

\begin{definition}\cite[Definition 2.5]{Ma2023}\label{Defn2.7}
Let $S \subseteq \mathbb{R}^r$, $\psi: \mathbb{R}^r \rightarrow \mathbb{R}$ be semidifferentiable at $\bar z \in S$.
 The second subderivative  of $\delta_{S}$ at $\bar z$ for $ {\rm d}\psi(\bar z)$   and $w$ is
$${\rm d}^2 \delta_S(\bar z; {\rm d}\psi(\bar z))(w):= \liminf\limits_{t \downarrow 0, w' \to w} \frac{\delta_{S}(\bar z+tw')-t{\rm d}\psi(\bar z)(w')}{\frac{1}{2}t^2},$$
where the sum of $\infty$ and $-\infty$ is interpreted as $\infty$.
\end{definition}

Next, we define a second-order constraint qualification that will be needed in the analysis of the second-order sufficient optimality conditions.

Let $\psi: \mathbb{R}^r \rightarrow \overline{\mathbb{R}}$, $\psi(\bar z)$ be finite and $\bar v, w\in  \mathbb{R}^r$. Recall that the function $\psi$ is said to be twice epi-differentiable at $\bar z$ for $\bar v$ if for any $w \in \mathbb{R}^r$ and any sequence $t_k\downarrow 0$ there exists a sequence $w_k\rightarrow w$ such that
\begin{equation*}
	{\rm d}^2\psi(\bar z;\bar v)(w) = \lim_{k\rightarrow \infty}  \frac{\psi(\bar{z}+t_kw_k)-\psi(\bar{z})-t_k\langle \bar{v}, w_k\rangle}{\frac{1}{2}t_k^2}.
\end{equation*}

We introduce the following concept, which is a slightly stronger variant of twice epi-differentiability. Note that when $\psi(x,y)$ does not depend on $x$, this notion of strong twice epi-differentiability reduces to the classical twice epi-differentiability \cite[Defintion13.6]{RoWe98}.

\begin{definition}[strongly twice epi-differentiable]\label{strongepi}  Let $\psi: {R}^{n+m} \rightarrow \bar{{R}}$, $\psi(\bar x,\bar y)$ be finite. We say 
$\psi$ is twice strongly epi-differentiable at $(\bar x,\bar y)$ for   $\xi=(\xi_1,\xi_2)$ if for all $(u,h)$ and  any $t_k\downarrow 0, u_k \rightarrow u$, there  exists {$h_k \rightarrow h$} such that 
 $${\rm d}^2\psi((\bar x,\bar y); (\xi_1,\xi_2))(u,h)=\lim_{k \rightarrow \infty } \frac{\psi(\bar{x}+t_ku_k, \bar y+t_k h_k)-\psi(\bar{x},\bar y)-t_k \langle (\xi_1,\xi_2),(u_k,h_k)\rangle }{\frac{1}{2}t_k^2}. $$
 \end{definition}

\section{Concepts of optimality for the minimax problem}\label{section-points}

In this section, we explore the concept of calm local minimax points and give several equivalent characterizations. These discussions offer deeper insights into the structure and interpretation of this notion. The concept was first introduced in \cite{Ma2023} for the simple minimax problem; here, we generalize it to minimax problems with coupled constraints, as formulated in \eqref{minimax}.

\begin{definition}[calm local minimax point]\label{calmmm}
A point $(\bar{x}, \bar{y}) \in X \times  Y(\bar x)$ is a calm local minimax point of problem \eqref{minimax}, if there exist a $\delta_0>0$ and a radius function $\tau: \mathbb{R}_{+} \rightarrow \mathbb{R}_{+}$ satisfying $\tau(\delta) \rightarrow 0$ as $\delta \downarrow 0$ and $\tau$ is calm at $0$, i.e., $\tau(\delta) \leq \kappa \delta$ for some $\kappa>0$, such that for any $\delta \in\left(0, \delta_0\right]$ and any $x\in X\cap \mathbb{B}_\delta(\bar x), y\in Y(\bar x)\cap \mathbb{B}_\delta(\bar y)$, we have
\begin{equation*}\label{localmm1}
f(\bar{x}, y) \leq f(\bar{x}, \bar{y}) \leq \max_{y' \in  Y(x)\cap\mathbb{B}_{\tau( \delta)}(\bar y)} f (x, y' ).
\end{equation*}
\end{definition}

We will need following definitions for the subsequent discussions.

\begin{definition}[inner calmness \text{\cite[Definition 2.2]{BGO19}}]\label{innercalm}
	Consider  a set-valued map $\Gamma: \mathbb{R}^n \rightrightarrows \mathbb{R}^m$.
	Given $\bar x\in X$ and $\bar y \in  \Gamma(\bar x)$, we say that the set-valued map $\Gamma$ is inner calm at $(\bar{x}, \bar{y})$ w.r.t. $X$ if there exist $\kappa >0$ and  $\delta_0>0$  such that
	$$\bar y \in \Gamma(x)+\kappa \|x-\bar x\| \mathbb{B} \quad \forall x\in \mathbb{B}_{\delta_0}(\bar x) \cap X,$$
	or equivalently (\cite[Lemma 2.2]{BGO19} or \cite[Definition 2.2]{Benko-2021}), if there exists $\kappa >0$ such that for any $x_k \rightarrow \bar x$ with $x_k \in X$ there exists a sequence $y_k$ satisfying $y_k \in \Gamma (x_k )$ and for sufficiently large $k$,
	$\|y_k-\bar y\| \leq \kappa \|x_k-\bar x\|.$
\end{definition}

{
\begin{definition}[inner semicontinuity \text{\cite[Definition 1.63]{M06}}]
Consider  a set-valued map $\Gamma: \mathbb{R}^n \rightrightarrows \mathbb{R}^m$. Given $\bar x\in X$ and $\bar y \in  \Gamma(\bar x)$, we say that the set-valued map $\Gamma$ is inner semicontinuous at $(\bar{x}, \bar{y})$ if for any $x_k \rightarrow \bar x$, there exists a sequence $y_k \in \Gamma\left(x_k\right)$ converging to $\bar y$.
\end{definition}
}

Throughout the paper, we assume that $Y(x)\cap\mathbb{B}_{\tau( \delta)}(\bar y) \not= \emptyset$ for some $\delta_0 >0$ and any $\delta \in\left(0, \delta_0\right], x \in X\cap \mathbb{B}_{\delta}(\bar x)$.  A sufficient condition for the above assumption is the inner calmness of the set-valued mapping $Y$ at $(\bar x,\bar y)$, i.e., for any $x_k \to \bar x$ with $x_k \in X$, there exists $\kappa >0$ and $y_k \in Y(x_k)$ such that $\|y_k -\bar y\| \leq \kappa \|x_k -\bar x\|$. By \cite[Theorem 5.2]{Bednarczuk2020}, when $Y(x):=\{y|g(x,y) \leq 0\}$ (where $g$ is locally Lipschitz continuous near $(\bar x, \bar y)$) is inner semicontinuous at $(\bar x,\bar y)$ and the RCRCQ holds at $(\bar x,\bar y)$, $Y$ is inner calm at $(\bar x,\bar y)$. If there exists a sequence $y_k \to \bar y$ such that $g(\bar x,y_k) <0$, then $Y$ is inner semicontinuous at $(\bar x,\bar y)$ \cite[Lemma 5.2]{Still2018}.

The subsequent equivalent definitions of the calm local minimax point will play a crucial role in establishing the relevant optimality conditions. The proof follows from \cite[Proposition 3.1, Lemma 3.1, Proposition 3.2]
{Ma2023}.

\begin{proposition}[equivalent definitions for the calm local minimax point]
\label{prop3.1}
Given $(\bar x,\bar y) \in X \times Y(\bar x)$. Then, the following concepts are equivalent.
\begin{itemize}
\item[(i)] $(\bar x,\bar y)$ is a calm local minimax point to problem \eqref{minimax}.
\item[(ii)] There exist constants $\delta_0>0$ and $\kappa>0$, such that for any $\delta\in(0, \delta_0]$ and any $x\in X\cap \mathbb{B}_{\delta}(\bar{x})$, $y\in Y(\bar x)\cap \mathbb{B}_{\delta}(\bar{y})$, we have
			\begin{equation}\label{equiv1}
				f(\bar{x}, y) \leq f(\bar{x},\bar{y})
				\leq V_{\kappa \delta}(x),
			\end{equation}
		where 
		\begin{equation}\label{localv}	V_{\epsilon_0}(x):=\max_{y' \in  Y(x)\cap\mathbb{B}_{\epsilon_0}(\bar y)} f (x, y' )\end{equation} for  $\epsilon_0 >0$ is the localized value function. 
\item[(iii)] There exist a constant $\delta_0>0$ and a radius function $\tau: \mathbb{R}_{+} \rightarrow \mathbb{R}_{+}$ satisfying $\tau(\delta) \rightarrow 0$ as $\delta \downarrow 0$ and $\tau$ is calm at $0$, such that for any $x\in X\cap \mathbb{B}_{\delta_0}(\bar x), y\in Y(\bar x)\cap \mathbb{B}_{\delta_0}(\bar y)$, we have
\begin{equation*}\label{equiv11}
f(\bar{x}, y) \leq f(\bar{x}, \bar{y}) \leq V_{\tau}(x),
\end{equation*}
where
\begin{equation}\label{localizedV}
V_{\tau}(x):=\max_{y' \in Y(x) \cap \mathbb{B}_{\tau(\|x-\bar x\|)}(\bar y)} f(x,y').
\end{equation}

\item[(iv)] $(\bar x,\bar y)$ is a local minimax point and the optimal solution mapping
\begin{equation}\label{localizedS}
S_{\tau}(x):=\argmax_{y' \in Y(x) \cap \mathbb{B}_{\tau(\|x-\bar x\|)}(\bar y)} f(x,y').
\end{equation}
is inner calm at $(\bar{x}, \bar{y})$.
\end{itemize}
\end{proposition}

Clearly, $V_{\tau}(\bar{x}) = f(\bar{x}, \bar{y})$ and $S_{\tau}(\bar{x}) = \{\bar{y}\}$. Based on Proposition \ref{prop3.1}, it follows that by imposing $\tau(\delta) \to 0$ as $\delta \downarrow 0$ in Definition \ref{calmmm}, the concept of calm local minimax points corresponds to considering the value function $V_{\tau}(x)$ with a radius function of the form $\tau(\delta) \equiv \kappa \delta$ for all $\delta > 0$. Specifically, $\left(\bar x, \bar y\right)$ is a calm local minimax point if and only if $\bar y$ is a local maximum point of $f\left(\bar x , \cdot\right)$  on $ Y(\bar x)$, and $\bar x$ is a local minimum point of $V_{\tau}(\cdot)$ on $X$ with the radius function being $\tau(\delta) \equiv \kappa\delta$ for any $\delta>0$. 

In the rest of this section, we show that existing discussions on optimality conditions for local minimax points in { \cite{DaiZh-2020}} are indeed discussing the calm local minimax points.

\begin{definition}[Jacobian uniqueness conditions \text{\label{jacobian} \cite[Definition 2.1]{DaiZh-2020}}]
Let $Y(x):=\{y \in \mathbb{R}^m|\varphi_i(x,y)=0, i=1,\dots, q_1; \varphi_i(x,y)=0,i=1,\dots, q_2\}$, $(\bar x,\bar y) \in X \times Y(\bar x)$, $f, \varphi$ be twice continuously differentiable around $(\bar x,\bar y)$.
 We say that Jacobian uniqueness conditions (without mentioning the multipliers)  of the maximization problem  $\max_{y\in Y(\bar x)} f(\bar x,y)$ are satisfied at 
 $\bar y$ 
 if  LICQ holds at $\bar y$, the Karush-Kuhn-Tucker condition holds at $\bar y$, the strict complementarity condition holds at $\bar y$, and the second-order sufficient optimality condition holds.
\end{definition}
By \cite[Lemma 2.1]{DaiZh-2020}, under the Jacobian uniqueness conditions, the solution mapping $S_{\epsilon_0}(x)$ is a single-valued map around the point $\bar x$ and  is twice continuously differentiable at $\bar x$. In fact not only the solution mapping but also the multiplier mapping a single-valued map around the point $\bar x$ and  is twice continuously differentiable. Hence  the Jacobian uniqueness condition is a very strong assumption. Under the Jacobian uniqueness condition, the concepts local minimax point and the calm local minimax point coincide.
\begin{proposition}\label{prop4.1}
Let $Y(x):=\{y \in \mathbb{R}^m|g(x,y)\leq 0, \varphi(x,y)=0\}$, $(\bar x,\bar y) \in X \times Y(\bar x)$, $f,g,\varphi$ be twice continuously differentiable around $(\bar x,\bar y)$, and $\left(\bar \beta_1, \bar \beta_2 \right) \in \mathbb{R}^{q_1} \times \mathbb{R}^{q_2}$. Suppose that the Jacobian uniqueness conditions hold at $(\bar x,\bar y,\bar \beta_1,\bar \beta_2)$. Then, the following concepts are equivalent.
\begin{itemize}
\item[(i)] $(\bar x,\bar y)$ is a calm local minimax point to problem \eqref{minimax}.


\item[(ii)] There exist $\delta_0>0, \epsilon_0>0$ such that for any $\delta\in(0, \delta_0]$ and any $x\in X\cap \mathbb{B}_{\delta}(\bar{x})$, $y\in Y(\bar x)\cap \mathbb{B}_{\delta}(\bar{y})$, we have
\begin{equation} \label{ii} f(\bar{x}, y) \leq f(\bar{x},\bar{y})
				\leq V_{\epsilon_0}(x)=\max_{y' \in  Y(x)\cap\mathbb{B}_{\epsilon_0}(\bar y)} f (x, y' ).
				\end{equation}

and  there exists a single-valued map $y(x)$ which is twice continuously differentiable at $\bar x$ such that $S_{\epsilon_0}(x):=\argmax_{y' \in  Y(x)\cap\mathbb{B}_{\epsilon_0}(\bar y)} f (x, y' )=\{y(x)\}$ for $x$ sufficiently close to  $\bar x$.

\end{itemize}
\end{proposition}
\begin{proof}

(i) $\Rightarrow$ (ii): \eqref{equiv1}  implies \eqref{ii} by fixing some $\delta>0$. By \cite[Lemma 2.1]{DaiZh-2020}, under the Jacobian uniqueness conditions, the solution mapping $S_{\epsilon_0}(x)$ has a single-valued twice continuously differentiable localization at $\bar x$. 

(ii) $\Rightarrow$ (i): Since under the Jacobian uniqueness condition, $S_{\epsilon_0}(x)$ has a single-valued twice continuously differentiable localization at $\bar x$,  $S_{\epsilon_0}(x)=\{y(x)\}$ for all $x$ sufficiently close to $\bar x$ and $y(x)$ is twice continuously differentiable at $\bar y$, there exist $\kappa >0$ and $\delta_1>0$ such that $\|y(x)-\bar y\|\leq \kappa \|x-\bar x\|$ for $x \in \mathbb{B}_{\delta_1}(\bar x)$. Let $\delta_0: =\min \{\delta_0,\delta_1,\epsilon_0/\kappa\}$. Then for any $\delta \in\left(0, \delta_0\right]$ and any $x\in X\cap \mathbb{B}_\delta(\bar x), y\in Y(\bar x)\cap \mathbb{B}_\delta(\bar y)$, we have
\begin{eqnarray*}
f(\bar{x}, y) & \leq  &  f(\bar{x}, \bar{y}) 
 \leq  \max_{y' \in  Y(x)\cap\mathbb{B}_{\epsilon_0}(\bar y)} f (x, y' ) = \max_{y' \in  Y(x)\cap\mathbb{B}_{\epsilon_0}(\bar y)\cap\mathbb{B}_{\kappa\|x-\bar x\|}(\bar y)} f (x, y' ) \\
 & \leq &\max_{y' \in  Y(x)\cap\mathbb{B}_{\kappa\delta}(\bar y)} f (x, y' ).
\end{eqnarray*}

By Proposition \ref{prop3.1}(ii), $(\bar x,\bar y)$ is a calm local minimax point.
 
\end{proof}

Note that the optimality conditions for local minimax points in \cite{DaiZh-2020} is  derived under the Jacobian uniqueness assumption and  the proof in \cite{DaiZh-2020} relies on the equivalent characterization given in Proposition~\ref{prop4.1}(ii).

\section{Optimality conditions for the minimax problem}\label{nonsmooth-section}

We begin with a first-order sensitivity analysis of the value function \( V_{\tau} \) defined in \eqref{localizedV}.

\begin{lemma}[first-order directional differentiability of the value function]\label{prop4.1}
Given $(\bar x,\bar y)\in \operatorname{gph}Y$. Suppose that $f$ is semidifferentiable at $(\bar x,\bar y)$, the set-valued mapping $Y$ is semidifferentiable at $\bar x$ for $\bar y$ and calm around $(\bar x,\bar y)$. Then, there exists $\kappa>0$ such that the value function $V_{\tau}$ defined in \eqref{localizedV} (with $\tau(\delta):=\kappa \delta$)
 is semidifferentiable (and thus directionally differentiable) at $\bar x$ and for any $u \in \mathbb{R}^n$,
$$V_{\tau}'(\bar x;u) = \max\limits_{h \in D Y(\bar x,\bar y)(u)} {\rm d} f(\bar x,\bar y)(u,h). $$
\end{lemma}
\begin{proof}
Let $u \in \mathbb{R}^n$ and $\tau(\delta):=\kappa \delta$ where $\kappa$ is given by Proposition \ref{prop2.1} (iii). There exist $t_k \downarrow 0, u_k \to u$ such that 
$${\rm d} V_{\tau} (\bar x)(u) = \lim\limits_{k \to \infty} \frac{V_{\tau}( \bar x+t_ku_k) - V_{\tau}(\bar x)}{t_k}.$$

Since $Y$ is semidifferentiable at $\bar x$ for $\bar y$ and calm around $(\bar x,\bar y)$, by Proposition \ref{prop2.1}, for any $h \in D Y(\bar x,\bar y)(u)$, there exists $h_k \to h$ such that $\bar y+t_kh_k \in Y(\bar x+t_ku_k)$ and $\|h_k\| \leq \kappa \|u_k\|$.

Then,
\begin{equation*}
\begin{aligned}
{\rm d} V_{\tau}(\bar x)(u)  & = \lim\limits_{k \to \infty} \frac{V_{\tau}\left(\bar{x}+t_ku_k\right)-V_{\tau}(\bar{x})}{t_k} \\
&\geq \liminf\limits_{k \to \infty} \frac{f\left(\bar{x}+t_k u_k, \bar{y}+t_k h_k\right)-f(\bar{x}, \bar{y})}{t_k} \\
&=  {\rm d} f(\bar x,\bar y)(u,h).
\end{aligned}
\end{equation*}
Thus, \begin{equation}\label{star}{\rm d} V_{\tau}(\bar x)(u)  \geq \sup_{h \in D Y(\bar x,\bar y)(u)} {\rm d} f(\bar x,\bar y)(u,h).\end{equation}

On the other hand, consider the sequence $t_k \downarrow 0, u_k \to u$ such that
$${\rm d}^+ V_{\tau} (\bar x)(u) = \lim\limits_{k \to \infty} \frac{V_{\tau}( \bar x+t_ku_k) - V_{\tau}(\bar x)}{t_k}.$$ 
Denote $x_k:= \bar x+t_ku_k$. Since $\tau(\delta)=\kappa \delta$, there exists a sequence
$$y_k \in S_{\tau}(x_k) = \argmax _{y \in Y(x_k)\cap \mathbb{B}_{\tau(\|x_k-\bar x\|)}(\bar{y})
	} f (x_k, y),$$
such that $y_k \in Y(x_k), \|y_k-\bar y\| \leq \kappa \|x_k-\bar x\|$ for sufficiently large $k$. Thus, by passing to a subsequence if necessary (without relabeling), there exists $h\in \mathbb{R}^m$ such that $h_k:=(y_k-\bar y)/t_k \to h$. By the definition of the graphical derivative, $h \in D Y(\bar x,\bar y)(u)$.
Hence we have
\begin{equation}\label{star2}
\begin{aligned}
{\rm d}^+ V_{\tau} (\bar x)(u)& =  \lim\limits_{k \to \infty} \frac{V_{\tau} (x_k)-V(\bar{x})}{t_k} \\
&= \lim\limits_{k \to \infty} \frac{f (x_k, y_k )-f(\bar{x}, \bar{y})}{t_k} \\
& = {\rm d} f(\bar x,\bar y)(u,h).   \qquad \mbox{ (by semidifferentiability of $f$)}  
\end{aligned}
\end{equation}

Thus, \begin{equation*}{\rm d}^+ V_{\tau}(\bar x)(u)  \leq \sup_{h \in D Y(\bar x,\bar y)(u)} {\rm d} f(\bar x,\bar y)(u,h).\end{equation*}

Then, combining with \eqref{star}, $V_{\tau}$ is directionally differentiable at $\bar x$ and 
\begin{equation*} V_{\tau}'(\bar x;u)  = \sup_{h \in D Y(\bar x,\bar y)(u)} {\rm d} f(\bar x,\bar y)(u,h) = \max_{h \in D Y(\bar x,\bar y)(u)} {\rm d} f(\bar x,\bar y)(u,h),\end{equation*}
where the second equality holds since \eqref{star2} holds for some $h \in D Y(\bar x,\bar y)(u)$.
\end{proof}

Now, we can give first-order optimality conditions as follows.  Even in the simple case where $Y(x)$ is independent of $x$, our results improved the corresponding results    (\cite[Theorem 4.1]{Ma2023}) in that  the subderivatives of $f$ does not need to satisfy the separation properties.

\begin{theorem}[first-order optimality conditions for the  couple-constrained minimax problem]\label{1st-non1}
Let $(\bar x, \bar y) \in X \times  Y(\bar x)$.
\begin{itemize}
\item[(a)] Suppose that either $Y(\bar x)$ is the whole space, or $f(\bar x,\cdot)$ is semidifferentiable at $\bar y$.  Suppose further that $f(x,\cdot)$ is continuous for any $x \in X$ near $\bar x$.
If $(\bar x, \bar y)$ is a calm local minimax point to the minimax problem \eqref{minimax}, then
\begin{eqnarray}\label{x-non1}
&& \sup_{h \in D Y(\bar x,\bar y)(u)} {\rm d}^+f(\bar{x}, \bar{y})(u,h) \geq 0 \quad \forall\, u \in T_X(\bar{x}),
\end{eqnarray}
\begin{equation}\label{y-non1}
{\rm d}_y^+ f(\bar{x}, \bar{y})(h) \leq 0\quad \forall\, h \in T_{Y(\bar x)}(\bar{y}).
\end{equation}
\item[(b)] Suppose that $f$ is semidifferentiable at $(\bar x,\bar y)$, the set-valued mapping $Y$ is semidifferentiable at $\bar x$ for $\bar y$ and calm around $(\bar x,\bar y)$, and
   \begin{align*}
  \sup_{h \in D Y(\bar x,\bar y)(u)} {\rm d}f(\bar{x}, \bar{y})(u,h) > 0 \quad &\forall\, u \in T_X(\bar{x}) \setminus \{0\}, \label{} \\
  {\rm d}_y f(\bar x,\bar y)(h)<0 \quad  &\forall\, h \in T_{ Y(\bar x)}(\bar y) \setminus \{0\}. \label{}
   \end{align*}
Then $(\bar x, \bar y)$ is a calm local minimax point to problem \eqref{minimax}.
\end{itemize}
\end{theorem}
\begin{proof}

(a) First, since $\bar y$ is a local maximum point of $f(\bar x,\cdot)$ on $Y(\bar x)$, \eqref{y-non1} or equivalently, {${\rm d}_y(-f)(\bar x,\bar y)(h)\geq 0$} follows from
\cite[Proposition 2.6 (i)]{Ma2023} (when $Y(\bar x)$ is the whole space) or \cite[Proposition 2.8 (i)]{Ma2023} (when $f(\bar x,\cdot)$ is semidifferentiable at $\bar y$).
We now prove (\ref{x-non1}). For this purpose, we let $u \in T_X(\bar{x})$. Then there exist $t_k \downarrow 0, u_k \to u$ such that $x_k:= \bar x+t_ku_k \in X$.

 Then since $(\bar x,\bar y)$ is a calm local minimax point, there exist $\kappa >0$ and a sequence
$$y_k \in S_{\tau}(x_k) = \argmax _{y \in Y(x_k)\cap \mathbb{B}_{\tau(\|x_k-\bar x\|)}(\bar{y})
	} f (x_k, y),$$
where $\tau(\cdot)$ is the function defined in the definition of the calm local minimax point,
such that $y_k \in Y(x_k), \|y_k-\bar y\| \leq \tau(\|x_k-\bar x\|)$ for sufficiently large $k$. Thus, by passing to a subsequence if necessary (without relabeling), there exists $h\in \mathbb{R}^m$ such that $h_k:=(y_k-\bar y)/t_k \to h$. By the definition of the graphical derivative, we have $h \in D Y(\bar{x}, \bar{y})(u)$.
Hence we have
\begin{eqnarray*}
\begin{aligned}
0 & \leq \limsup\limits_{k \to \infty} \frac{f (x_k, y_k )-f(\bar{x}, \bar{y})}{t_k}  \qquad \mbox{ by (\ref{equiv1})}\\
& \leq \limsup\limits_{{t \downarrow 0} \atop {(u',h') \to (u,h)}} \frac{f (\bar{x}+tu', \bar{y}+th' )-f(\bar{x}, \bar{y})}{t} \\
& = {\rm d}^+ f(\bar x,\bar y)(u,h).
\end{aligned}
\end{eqnarray*}

(b) By \cite[Proposition 2.8 (ii)]{Ma2023}, Lemma \ref{prop4.1}, and Proposition \ref{prop3.1} (iii), $(\bar x, \bar y)$ is a calm local minimax point to problem \eqref{minimax}.
\end{proof}

Given $(\bar x,\bar y)\in \operatorname{gph}Y$. Suppose that $f$ is semidifferentiable at $(\bar x,\bar y)$, $Y$ is semidifferentiable at $\bar x$ for $\bar y$ and calm around $(\bar x,\bar y)$. Based on Lemma \ref{prop4.1}, we can define the critical cone for the minimization problem $\min_{x\in X} V_\tau (x)$ at $\bar x$ as
 $$C_{\min}(\bar x; \bar y):=T_{ X}(\bar x) \cap \{u|\max_{h'\in DY(\bar x,\bar y)(u)}{\rm d} f(\bar x,\bar y)(u,h') \leq  0\}.$$ Moreover by \eqref{x-non1}, if $\bar x$ is a local solution of the minimization problem $\min_{x\in X} V_\tau (x)$, then
 the critical cone for the minimization problem $\min_{x\in X} V_\tau (x)$ becomes
 $$C_{\min}(\bar x; \bar y)=T_{ X}(\bar x) \cap \{u|\max_{h'\in DY(\bar x,\bar y)(u)}{\rm d} f(\bar x,\bar y)(u,h') =  0\}.$$
Similarly, we define the critical cone for the maximization problem as
$$C_{\max}(\bar y;\bar x):=T_{ Y(\bar x)}(\bar y) \cap \{h'|{\rm d}_y f(\bar x,\bar y)(h') \geq  0\}.$$
Moreover by (\ref{y-non1}), if $\bar y$ is  an optimal solution for the maximization problem $\max_{y\in Y(\bar x)} f(\bar x, y)$, then 
$$C_{\max}(\bar y;\bar x)=T_{ Y(\bar x)}(\bar y) \cap \{h'|{\rm d}_y f(\bar x,\bar y)(h') =  0\}.$$
Define
$$C(\bar x,\bar y;u):=D Y(\bar x,\bar y)(u)\cap \{h'| {\rm d}f(\bar x,\bar y)(u,h')=0\}.$$

Next, we give  second-order optimality conditions for the minimax problem \eqref{minimax} for the general case. Even in the simple case where $Y(x)$ is independent of $x$, our results improved the corresponding results    (\cite[Theorem 4.2]{Ma2023}) in that  the subderivatives of $f$ does not need to satisfy the separation properties. Note that the strongly twice epi-differentiability assumption reduces to the twice epi-differentiability required in \cite[Theorem 4.2]{Ma2023}.


\begin{theorem}[Second-order optimality conditions for the couple-constrained minimax problem]\label{main1}
Let $(\bar x, \bar y)\in X \times Y(\bar x)$. Suppose that $f$ is twice semidifferentiable at $(\bar x, \bar y)$, $Y$ is semidifferentiable at $\bar x$ for $\bar y$ and calm around $(\bar x,\bar y)$. 


\begin{itemize}
\item[(a)] Let $(\bar x,\bar y)$ be a calm local minimax point to problem \eqref{minimax}. 
 Then the first-order necessary optimality conditions 
 \begin{align}
  \sup\limits_{h \in D Y(\bar x,\bar y)(u)} {\rm d} f(\bar x,\bar y)(u,h) \geq 0 \quad &\text { for all } u\in T_{X}(\bar x), \label{firstordercond11}\\
{\rm d}_y f(\bar x,\bar y)(h)\leq 0 \quad  &\text { for all } h \in T_{ Y(\bar x)}(\bar y)\label{firstordercond21}
   \end{align}
hold. 
The second-order necessary condition for the maximum problem $\max_{y\in Y(\bar x)} f(\bar x, y)$ holds at $\bar y$, i.e., for any $h\in C_{\max}(\bar y;\bar x)$, we have
\begin{equation*}\label{necelower1}
{\rm d}^2_{yy} f(\bar x,\bar y)(h) -{\rm d}^2 \delta_{Y(\bar x)}(\bar y;{\rm d}_yf(\bar x,\bar y))(h) \leq 0.
\end{equation*}
 Moreover, suppose that for any $u\in C_{\min}(\bar x; \bar y)$, there exist $$-v^u\in N_X(\bar x)\cap \{u\}^\perp, (\xi_1^u,\xi_2^u) \in N_{\operatorname{gph}Y}(\bar x,\bar y)$$ such that  $\langle (v^u,0)+(\xi_1^u,\xi_2^u), (u',h') \rangle = {\rm d} f(\bar x,\bar y)(u',h')$ for any $(u',h') \in \mathbb{R}^n \times \mathbb{R}^m$, that the value ${\rm d}^2 \delta_X(\bar x;-v^u)(u)$ is finite and ${\rm d}^2 \delta_{\operatorname{gph}Y}((\bar x,\bar y);{\rm d}f(\bar x,\bar y))(u,h)$ is finite for any $h \in C(\bar x,\bar y;u)$. Then 
 for any $u\in C_{\min}(\bar x; \bar y)$, there exists
 $h\in C(\bar x,\bar y;u)$ 
  such that
 \begin{equation*}\label{eqn4.111}
{\rm d}^2 f(\bar x,\bar y)(u,h) + {\rm d}^2 \delta_X(\bar x;-v^u)(u) -{\rm d}^2 \delta_{\operatorname{gph}Y}((\bar x,\bar y);{\rm d}f(\bar x,\bar y))(u,h)  \geq 0.
\end{equation*}

\item[(b)] Suppose that the first-order necessary optimality conditions (\ref{firstordercond11})-(\ref{firstordercond21})  hold and the second-order sufficient condition for  problem $\max_{y\in Y(\bar x)} f(\bar x, y)$ holds at $\bar y$, i.e., for any $h\in T_{ Y(\bar x)}(\bar y) \cap \{h'|{\rm d}_y f(\bar x,\bar y)(h') =  0\} \setminus \{0\}$,
\begin{equation}\label{sufflower}
{\rm d}^2_{yy} f(\bar x,\bar y)(h) -{\rm d}^2 \delta_{Y(\bar x)}(\bar y;{\rm d}_yf(\bar x,\bar y))(h)<0.
\end{equation}
Moreover, suppose  that for any $$u\in  T_{ X}(\bar x) \cap \{u|\max_{h'\in DY(\bar x,\bar y)(u)}{\rm d} f(\bar x,\bar y)(u,h') =  0\},$$ there exist $-v^u\in N_X(\bar x)\cap \{u\}^\perp, (\xi_1^u,\xi_2^u) \in N_{\operatorname{gph}Y}(\bar x,\bar y)$ such that  $$\langle (v^u,0)+(\xi_1^u,\xi_2^u), (u',h') \rangle = {\rm d} f(\bar x,\bar y)(u',h') \qquad \forall (u',h') \in \mathbb{R}^n \times \mathbb{R}^m,$$  that the value ${\rm d}^2 \delta_X(\bar x;-v^u)(u)$ is finite and ${\rm d}^2 \delta_{\operatorname{gph}Y}((\bar x,\bar y);{\rm d}f(\bar x,\bar y))(u,h)$ is finite for any $h \in C(\bar x,\bar y;u)$ and  assume that \( \delta_{\operatorname{gph} Y} \) is strongly twice epi-differentiable at \( (\bar{x}, \bar{y}) \) for any  \( (\xi_1^u, \xi_2^u) \) satisfying the aforementioned assumptions.
If for any $u\in  T_{ X}(\bar x) \cap \{u|\max_{h'\in DY(\bar x,\bar y)(u)}{\rm d} f(\bar x,\bar y)(u,h') =  0\} \setminus \{0\}$, 
\begin{equation}\label{cxlower3-non-new1}
\sup_{h \in C(\bar x,\bar y;u)}  \left\{ {\rm d}^2 f(\bar x,\bar y)(u,h) + {\rm d}^2 \delta_X(\bar x;-v^u)(u) -{\rm d}^2 \delta_{\operatorname{gph}Y}((\bar x,\bar y);{\rm d}f(\bar x,\bar y))(u,h) \right\} >0,
\end{equation}
then $(\bar x,\bar y)$ is a calm local minimax point to problem \eqref{minimax} and  the  following second-order growth condition holds: there exist $\delta_0>0$,   $\varepsilon>0, \mu > 0, \kappa>0$
 such that  for any $x \in X \cap \mathbb{B}_{\delta_0}(\bar x)$ and $y \in Y(\bar x)\cap \mathbb{B}_{\delta_0}(\bar y)$, we have
\begin{equation*}
f(\bar x, y)+\varepsilon\|y-\bar y\|^2  \leq f(\bar x, \bar y) \leq  \max _{y^{\prime}  \in  Y(x)\cap \mathbb{B}_{\kappa\|x-\bar x\|}(\bar y)} f(x, y^{\prime}) - \mu \|x-\bar x\|^2.
\end{equation*}
\end{itemize}
\end{theorem}
\begin{proof}

(a) First, by \cite[Proposition 2.9 (i)]{Ma2023}, we have the first- and second-order optimality conditions for the maximization problem.

Second, since $(\bar x,\bar y)$ is a calm local minimax point to problem \eqref{minimax}, by Proposition \ref{prop3.1}, there exist a $\delta_0 >0$ and a function $\tau: \mathbb{R}_{+} \rightarrow \mathbb{R}_{+}$ which is calm at $0$, such that for any $x \in X\cap \mathbb{B}_{\delta_0} (\bar x)$, we have
\begin{equation}
 f(\bar{x}, \bar{y}) \leq \max _{y^{\prime}  \in  Y(x)\cap \mathbb{B}_{\tau(\|x-\bar x\|)}(\bar y)} f\left(x, y^{\prime}\right) .\label{eqn4.211}
\end{equation}
Pick  $u\in C_{\min}(\bar x; \bar y)$ and $v^u$ given in the assumption.  By definition of the second subderivative, there exist $t_k\downarrow 0$, $u_k\to u$ such that $x_k:=\bar x+t_ku_k \in X$ and
\begin{equation}\label{finitelim}
{\rm d}^2 \delta_X(\bar x;-v^u)(u) = \lim\limits_{k\to \infty}\frac{t_k \langle v^u,u_k \rangle}{\frac{1}{2}t_k^2}.\end{equation}

 By the assumptions, we have $-v^u\in N_X(\bar x)\cap \{u\}^\perp, (\xi_1^u,\xi_2^u) \in N_{\operatorname{gph}Y}(\bar x,\bar y)$ such that   $\langle (v^u,0)+(\xi_1^u,\xi_2^u), (u',h') \rangle = {\rm d} f(\bar x,\bar y)(u',h')$ for any $(u',h') \in \mathbb{R}^n \times \mathbb{R}^m$, and   the  limit (\ref{finitelim})  is finite.

Since $\tau $ is calm,  there exist $\kappa>0$ and
\begin{equation*}
y_k \in \argmax _{y^{\prime}  \in  Y(x_k)\cap \mathbb{B}_{\tau(\|x_k-\bar x\|)}(\bar y)} f\left(x_k, y^{\prime}\right),
\end{equation*}
 such that $y_k \in Y(x_k), \|y_k-\bar y\| \leq \kappa \|x_k-\bar x\|$. Thus, by passing to a subsequence if necessary (without relabeling), there exists $h\in \mathbb{R}^m$ such that $h_k:=(y_k-\bar y)/t_k \to h$. By the definition of the graphical derivative, we have $h\in D Y(\bar x,\bar y)(u)$. 

Thus,
\begin{eqnarray*}\label{}
\begin{aligned}
0 & \leq \limsup\limits_{k \to \infty} \frac{f\left(\bar{x}+t_k u_k, \bar{y}+t_kh_k\right)-f(\bar{x}, \bar{y})}{\frac{1}{2}t_k^2} \quad \mbox{ by } \eqref{eqn4.211}\\
& =\limsup\limits_{k \to \infty} \frac{f\left(\bar{x}+t_k u_k, \bar{y}+t_kh_k\right)-f(\bar{x}, \bar{y})-t_k{\rm d}f(\bar{x}, \bar{y})(u_k,h_k)+t_k{\rm d}f(\bar{x}, \bar{y})(u_k,h_k)}{\frac{1}{2}t_k^2}  \\
& = {\rm d}^2 f(\bar x,\bar y)(u,h) + \lim\limits_{k \to \infty} \frac{t_k \langle v^u,u_k \rangle}{\frac{1}{2}t_k^2} + \limsup\limits_{k \to \infty} \frac{t_k \langle (\xi_1^u,\xi_2^u),(u_k,h_k) \rangle}{\frac{1}{2}t_k^2} \\
& \leq {\rm d}^2 f(\bar x,\bar y)(u,h) + {\rm d}^2 \delta_X(\bar x;-v^u)(u) + \limsup\limits_{{t \downarrow 0,(u',h') \to (u,h)} \atop {\bar y+th'\in Y(\bar x+tu')}} \frac{t \langle (\xi_1^u,\xi_2^u),(u',h') \rangle}{\frac{1}{2}t^2} \\
& = {\rm d}^2 f(\bar x,\bar y)(u,h) + {\rm d}^2 \delta_X(\bar x;-v^u)(u) - \liminf\limits_{{t \downarrow 0,(u',h') \to (u,h)} \atop {\bar y+th'\in Y(\bar x+tu')}} \frac{-t \langle (\xi_1^u,\xi_2^u),(u',h') \rangle}{\frac{1}{2}t^2} \\
& = {\rm d}^2 f(\bar x,\bar y)(u,h) + {\rm d}^2 \delta_X(\bar x;-v^u)(u) -{\rm d}^2 \delta_{\operatorname{gph}Y}((\bar x,\bar y);(\xi_1^u,\xi_2^u))(u,h),
\end{aligned}
\end{eqnarray*}
where the second equality follows from the assumption that
$\langle (v^u,0)+(\xi_1^u,\xi_2^u), (u_k,h_k) \rangle = {\rm d} f(\bar x,\bar y)(u_k,h_k).$

It now remains to show that $h\in C(\bar x,\bar y,u)$. Since   $\sup\limits_{h' \in D Y(\bar x,\bar y)(u)} {\rm d} f(\bar x,\bar y)(u,h') = 0$, we have ${\rm d} f(\bar x,\bar y)(u,h) \leq 0$. On the other hand,
\begin{eqnarray*}\label{}
\begin{aligned}
0  & \leq  \lim\limits_{k \to \infty} \frac{f(x_k,y_k)-f(\bar x,\bar y)}{t_k} = {\rm d} f(\bar x,\bar y)(u,h).
\end{aligned}
\end{eqnarray*}
Thus, ${\rm d} f(\bar x,\bar y)(u,h) = 0$. Moreover by the assumption, we have
$${\rm d} f(\bar x,\bar y)(u,h)=\langle v^u, u\rangle +\langle (\xi_1^u,\xi_2^u),(u,h) \rangle=\langle (\xi_1^u,\xi_2^u),(u,h) \rangle.$$
Therefore $h\in C(\bar x,\bar y,u)$.

(b) Since $f$ is twice semidifferentiable at $(\bar x,\bar y)$, by (\ref{eqn2.2new}) and (\ref{twice-epi}), we have
$$-{\rm d}_y f(\bar x,\bar y)(h)={\rm d}_y (-f)(\bar x,\bar y)(h), \quad -{\rm d}_{yy}^2 f(\bar x,\bar y)(h)={\rm d}_{yy}^2 (-f)(\bar x,\bar y)(h).$$ By \cite[Proposition 2.9]{Ma2023}, the second-order sufficient condition for the maximum problem implies that $\bar y$ is a local maximizer of $f(\bar x,\cdot)$ on $ Y(\bar x)$ and the second-order growth condition holds. Thus, there exist $\delta_0>0, \varepsilon>0$ such that for any $\delta \in (0,\delta_0], y\in Y(\bar x)$ satisfying $\|y-\bar{y}\| \leq \delta$, we have  $f(\bar x,y) +\varepsilon\|y-\bar y\|^2\leq f(\bar x,\bar y)$. Let $\tau(\delta):=\delta$, then $\tau(\delta) \rightarrow 0$ as $\delta \downarrow 0$.
Since
$V_\delta(x)=\max_{y  \in  Y(x)\cap \mathbb{B}_\delta(\bar y)}f(x,y),$
we have
\begin{equation}\label{eqn4.13} V_\delta( x)\geq f(x,y) \ \ \forall (x,y)\in \mathbb{B}_\delta(\bar x,\bar y)\cap (X\times Y(x)), \qquad \mbox{ and } V_\delta(\bar x)=f(\bar x,\bar y).\end{equation}
We break the rest proof for (b) into two steps.

{\bf Step 1:} We show that for any fixed $u \in C_{\min}(\bar x; \bar y)$, with $-v^u$ and $(\xi_1^u, \xi_2^u)$ as the vectors given in the assumptions, the following holds for any $\delta \in (0, \delta_0]$:
\begin{equation}\label{lowerbd2}
	\begin{aligned}
		&{\rm d}^2 (V_{\delta}+\delta_X)(\bar x;0)(u) \\
		\geq& \sup_{h \in C_{y}(\bar x,\bar y,u) } \left\{ {\rm d}^2 f(\bar x,\bar y)(u,h) + {\rm d}^2 \delta_X(\bar x;-v^u)(u) -{\rm d}^2 \delta_{\operatorname{gph}Y}((\bar x,\bar y);(\xi_1^u,\xi_2^u))(u,h) \right\}.
	\end{aligned}
\end{equation}
Since $u \in T_X(\bar x) $, by definition of the second subderivative, there exist $t_k\downarrow 0$, $u_k\to u$ such that  $\bar x+t_ku_k \in X$ and
\begin{equation}\label{main-1}
	{\rm d}^2 (V_{\delta}+\delta_X)(\bar x;0)(u)
	= \lim\limits_{k \to \infty} \frac{V_{\delta}(\bar{x}+t_ku_k)-V_{\delta}(\bar{x})}{\frac{1}{2}t_k^2}.
\end{equation}

Since $\delta_{\operatorname{gph} Y}$ is strongly twice epi-differentiable at $(\bar x,\bar y)$ for $(\xi_1^u,\xi_2^u)$, for any $h \in C(\bar x,\bar y,u)$,  there exists a sequence $ h_k\rightarrow h$ such that 
$$ {\rm d}^2 \delta_{\operatorname{gph} Y} ((\bar x,\bar y); {\rm d}f(\bar x,\bar y))(u,h)=\lim_{k\rightarrow \infty}\frac{\delta_{\operatorname{gph}Y}(\bar x+t_ku_k,\bar y+t_kh_k)  - t_k \langle (\xi_1^u,\xi_2^u), (u_k,h_k)\rangle}{\frac{1}{2}{t_k}^2} .$$
Since ${\rm d}^2 \delta_{\operatorname{gph} Y} ((\bar x,\bar y); {\rm d}f(\bar x,\bar y))(u,h)$ is finite by the assumption, we have $\bar y+t_kh_k \in Y(\bar x+t_ku_k)$. Then
\begin{equation}\label{eqn4.16}-{\rm d}^2 \delta_{\operatorname{gph} Y}((\bar x,\bar y);{\rm d}f(\bar x,\bar y))(u,h)=\lim\limits_{k \to \infty}\frac{t_k\langle (\xi_1^u,\xi_2^u), (u_k,h_k)\rangle}{\frac{1}{2}{t_k}^2}.\end{equation}
Hence we have
\begin{eqnarray*}\label{}
	\begin{aligned}
		& \quad\ {\rm d}^2 (V_{\delta}+\delta_X)(\bar x;0)(u)  = \liminf\limits_{k \to \infty} \frac{V_{\delta}\left(\bar{x}+t_ku_k\right)-V_{\delta}(\bar{x})}{\frac{1}{2} t_k^2}  \quad \mbox{ by } \eqref{main-1}\\
		&\geq \liminf\limits_{k \to \infty} \frac{f\left(\bar{x}+t_k u_k, \bar{y}+t_k h_k\right)-f(\bar{x}, \bar{y})}{\frac{1}{2} t_k^2}\quad \mbox{ by } \eqref{eqn4.13}\\
		&= \liminf\limits_{k \to \infty} \frac{f\left(\bar{x}+t_k u_k, \bar{y}+t_k h_k\right)-f(\bar{x}, \bar{y})-t_k{\rm d}f(\bar{x}, \bar{y})(u_k,h_k)}{\frac{1}{2} t_k^2} \\
		& \quad\quad\quad\quad + \frac{t_k \langle v^u,u_k \rangle}{\frac{1}{2}t_k^2} +\frac{t_k \langle (\xi_1^u,\xi_2^u), (u_k,h_k)\rangle}{\frac{1}{2}t_k^2}  \\
		&\geq {\rm d}^2 f(\bar x,\bar y)(u,h) + \liminf_{k \to \infty} \frac{t_k \langle v^u,u_k \rangle}{\frac{1}{2}{t_k}^2} +  \liminf_{k \to \infty}\frac{t_k \langle (\xi_1^u,\xi_2^u), (u_k,h_k)\rangle}{\frac{1}{2}{t_k}^2}    \\
&\geq {\rm d}^2 f(\bar x,\bar y)(u,h) + {\rm d}^2 \delta_X(\bar x;-v^u)(u) -{\rm d}^2 \delta_{\operatorname{gph}Y}((\bar x,\bar y);{\rm d}f(\bar x,\bar y))(u,h),
	\end{aligned}
\end{eqnarray*}
where the second inequality holds since ${\rm d}^2 \delta_{\operatorname{gph} Y}((\bar x,\bar y);{\rm d}f(\bar x,\bar y))(u,h)$ is finite.
Thus (\ref{lowerbd2}) holds.

{\bf Step 2:} We show that for any $\delta \in (0,\delta_0]$ and $x \in X$ satisfying $\|x-\bar x\| \leq \delta$, we have
\begin{equation}\label{second-order-growth} \max _{y^{\prime}  \in  Y(x)\cap \mathbb{B}_\delta(\bar y)} f\left(x, y^{\prime}\right) - f(\bar x,\bar y) \geq \beta \|x-\bar x\|^2\end{equation}
for some $\beta > 0$.

To the contrary, suppose that for some $\delta \in (0,\delta_0]$ and $x_k \in X$  with $\|x_k-\bar x\| \leq \delta$,
\begin{equation}\label{contradiction2}
	\max _{y^{\prime}  \in  Y(x_k)\cap \mathbb{B}_\delta(\bar y)} f\left(x_k, y^{\prime}\right)-f(\bar x,\bar y) \leq o(t_k^2),
\end{equation}
where $t_k:= \|x_k-\bar x\|$. Let $u_k:=(x_k-\bar x)/\|x_k-\bar x\|$, we have $t_k\downarrow 0$ and $\|u_k\|=1$. By passing to a subsequence if necessary, we may assume that $u_k \to u$ with $\|u\|=1$. We have $u \in T_{X}(\bar x) \setminus \{0\}$.

The assumed first-order optimality condition gives us $\sup\limits_{h \in D Y(\bar x,\bar y)(u)} {\rm d} f(\bar x,\bar y)(u,h) \geq 0$.
If $\sup\limits_{h \in D Y(\bar x,\bar y)(u)} {\rm d} f(\bar x,\bar y)(u,h) >0$,
then there exists $h \in D Y(\bar x,\bar y)(u)$ such that ${\rm d} f(\bar x,\bar y)(u,h) >0$ and there exists $h_k \to h$ such that $y_k:=\bar y+t_kh_k \in Y(\bar x+t_ku_k)$ and $\|h_k\| \leq \kappa \|u_k\|$ for any $k$ (by the semidifferentiability of $Y$). Then,
\begin{eqnarray*}\max _{y^{\prime}  \in  Y(x_k)\cap \mathbb{B}_{\delta}(\bar y)} f\left(x_k, y^{\prime}\right)-f(\bar x,\bar y) & \geq & f(x_k,y_k)-f(\bar x,\bar y)\\
& \geq & t_k {\rm d} f(\bar x,\bar y)(u,h) +o(t_k) > o(t_k) \geq o(t_k^2),\end{eqnarray*}
which is a contradiction to \eqref{contradiction2}.

If $\sup\limits_{h \in D Y(\bar x,\bar y)(u)} {\rm d} f(\bar x,\bar y)(u,h) = 0$, by \eqref{lowerbd2} and (\ref{main-1}), we have
\begin{equation}\label{theta}
	\begin{aligned}
		\max _{y^{\prime}  \in  Y(x_k)\cap \mathbb{B}_{\delta}(\bar y)} f\left(x_k, y^{\prime}\right)-f(\bar x,\bar y) & = V_{\delta}(x_k) - V_{\delta}(\bar x)  \geq \frac{1}{2}t_k^2 \theta(\bar x,\bar y, u)+o(t_k^2),
	\end{aligned}
\end{equation}
where 
\begin{eqnarray*}
\lefteqn{\theta(\bar x,\bar y, u):=}\\
&&  \sup_{h \in C(\bar x,\bar y,u)} \left\{ {\rm d}^2 f(\bar x,\bar y)(u,h) + {\rm d}^2 \delta_X(\bar x;-v^u)(u) -{\rm d}^2 \delta_{\operatorname{gph}Y}((\bar x,\bar y);{\rm d}f(\bar x,\bar y))(u,h) \right\}.\end{eqnarray*}
It follows from \eqref{cxlower3-non-new1} that $\theta(\bar x,\bar y, u)>0$. Hence we have a contradiction to \eqref{contradiction2} and consequently (\ref{second-order-growth}) holds.

Combining with the fact that $\bar y$ is a  maximizer of $f(\bar x,\cdot)$ on $ Y\cap \mathbb{B}_\delta(\bar y)$, it follows that
$$f(\bar x, y) \leq f(\bar x,\bar y)\leq \max _{y^{\prime}  \in  Y(x)\cap \mathbb{B}_\delta(\bar y)} f\left(x, y^{\prime}\right) \quad \forall (x,y) \in \mathbb{B}_\delta(\bar x,\bar y) \cap (X\times Y(x)).$$
Thus, $(\bar x, \bar y)$ is a calm local minimax point to problem \eqref{minimax}.
\end{proof}

In what follows, we show that the strongly twice epi-differentiability assumption can be replaced by the twice directional differentiability of  the localized value function $V_\delta$.

\begin{corollary}[Sufficient optimality conditions without the strongly twice epidifferentiability]\label{sufficient2}
	Let $(\bar x, \bar y)\in X \times Y(\bar x)$. Suppose that $f$ is twice semidifferentiable at $(\bar x, \bar y)$, $Y$ is semidifferentiable at $\bar x$ for $\bar y$ and calm around $(\bar x,\bar y)$. Suppose that the first-order necessary optimality conditions \eqref{firstordercond11}-\eqref{firstordercond21}  hold and the second-order sufficient condition for  problem $\max_{y\in Y(\bar x)} f(\bar x, y)$ \eqref{sufflower} holds at $\bar y$.
		If there exists $\delta_0$ such that  for any $\delta \in (0,\delta_0]$, the value function $V_{\delta}$ is Lipschitz continuous at $\bar x$,  and is twice directional differentiable at $\bar x$ in directions  $u\in  T_{ X}(\bar x) \cap \{u|\max_{h'\in DY(\bar x,\bar y)(u)}{\rm d} f(\bar x,\bar y)(u,h') =  0\}$, and for any  $u\in  T_{ X}(\bar x) \cap \{u|\max_{h'\in DY(\bar x,\bar y)(u)}{\rm d} f(\bar x,\bar y)(u,h') =  0\} \setminus \{0\}$, there exists $v^u$ such that $V^{\prime}_{\delta}(\bar x;u') = \langle v^u, u' \rangle$ for all $u'$, and
		\begin{equation}\label{cor4.1suf}
			V^{''}_{\delta}(\bar x;u) + {\rm d}^2 \delta_X(\bar x;-v^u)(u) > 0.
			\end{equation}
		Then $(\bar x,\bar y)$ is a calm local minimax point to problem \eqref{minimax} and  the  second-order growth condition holds.
\end{corollary}
\begin{proof}
The maximization with respect to $\bar y$ can be established in a manner similar to Theorem \ref{main1} (b), by using \eqref{firstordercond21} and \eqref{sufflower}. 
	Next, we consider the minimization with respect to $\bar x$. 
	Analogous to Lemma \ref{prop4.1}, one can show that  for any $u \in \mathbb{R}^n$,
	$$V^{\prime}_{\delta}(\bar x;u) = \sup\limits_{h \in D Y(\bar x,\bar y)(u)} {\rm d} f(\bar x,\bar y)(u,h).$$

Next, we show that for any $\delta \in (0,\delta_0]$ and $x \in X$ satisfying $\|x-\bar x\| \leq \delta$, we have
\begin{equation}\label{second-order-growth3} \max _{y^{\prime}  \in  Y(x)\cap \mathbb{B}_\delta(\bar y)} f\left(x, y^{\prime}\right) - f(\bar x,\bar y) \geq \beta \|x-\bar x\|^2\end{equation}
for some $\beta > 0$.

To the contrary, suppose that for some $\delta \in (0,\delta_0]$ and $x_k \in X$  with $\|x_k-\bar x\| \leq \delta$,
\begin{equation}\label{contradiction3}
	\max _{y^{\prime}  \in  Y(x_k)\cap \mathbb{B}_\delta(\bar y)} f\left(x_k, y^{\prime}\right)-f(\bar x,\bar y) \leq o(t_k^2),
\end{equation}
where $t_k:= \|x_k-\bar x\|$. Let $u_k:=(x_k-\bar x)/\|x_k-\bar x\|$, we have $t_k\downarrow 0$ and $\|u_k\|=1$. By passing to a subsequence if necessary, we may assume that $u_k \to u$ with $\|u\|=1$. We have $u \in T_{X}(\bar x) \setminus \{0\}$.

The assumed first-order optimality condition gives us $\sup\limits_{h \in D Y(\bar x,\bar y)(u)} {\rm d} f(\bar x,\bar y)(u,h) \geq 0$.
If $\sup\limits_{h \in D Y(\bar x,\bar y)(u)} {\rm d} f(\bar x,\bar y)(u,h) >0$,
then there exists $h \in D Y(\bar x,\bar y)(u)$ such that ${\rm d} f(\bar x,\bar y)(u,h) >0$ and there exists $h_k \to h$ such that $y_k:=\bar y+t_kh_k \in Y(\bar x+t_ku_k)$ and $\|h_k\| \leq \kappa \|u_k\|$ for any $k$ (by the semidifferentiability of $Y$). Then,
\begin{eqnarray*}\max _{y^{\prime}  \in  Y(x_k)\cap \mathbb{B}_{\delta}(\bar y)} f\left(x_k, y^{\prime}\right)-f(\bar x,\bar y) & \geq & f(x_k,y_k)-f(\bar x,\bar y)\\
&  \geq & t_k {\rm d} f(\bar x,\bar y)(u,h) +o(t_k) > o(t_k) \geq o(t_k^2),\end{eqnarray*}
which is a contradiction to \eqref{contradiction3}.

If $\sup\limits_{h \in D Y(\bar x,\bar y)(u)} {\rm d} f(\bar x,\bar y)(u,h) = 0$, since $V_{\delta}$ is Lipschitz continuous at $\bar x$ and twice directionally differentiable at $\bar x$ in direction $u \in T_X(\bar x) $, by definition of the second subderivative, there exist $t_j\downarrow 0$, $u_j\to u$ such that  $\bar x+t_ju_j \in X$ and
\begin{eqnarray*}\label{}
	\begin{aligned}
		& \quad\ {\rm d}^2 (V_{\delta}+\delta_X)(\bar x;0)(u)  = \liminf\limits_{j \to \infty} \frac{V_{\delta}\left(\bar{x}+t_ju_j\right)-V_{\delta}(\bar{x})}{\frac{1}{2} t_j^2}  \\
		&= \liminf\limits_{j \to \infty} \left\{ \frac{V_{\delta}\left(\bar{x}+t_ju_j\right)-V_{\delta}(\bar{x})-t_jV^{\prime}_{\delta}(\bar x;u_j)}{\frac{1}{2} t_j^2} + \frac{t_jV^{\prime}_{\delta}(\bar x;u_j)}{\frac{1}{2} t_j^2} \right\} \\
		&=V^{''}_{\delta}(\bar x;u) + \liminf_{j \to \infty} \frac{t_jV^{\prime}_{\delta}(\bar x;u_j)}{\frac{1}{2}{t_j}^2}    \\
		&\geq V^{''}_{\delta}(\bar x;u) + {\rm d}^2 \delta_X(\bar x;-v^u)(u) \\
		& >0 \quad \mbox{ by } \eqref{cor4.1suf}.\\
	\end{aligned}
\end{eqnarray*}
Then,
\begin{equation*}
	\begin{aligned}
		\max _{y^{\prime}  \in  Y(x_k)\cap \mathbb{B}_{\delta}(\bar y)} f\left(x_k, y^{\prime}\right)-f(\bar x,\bar y) & = V_{\delta}(x_k) - V_{\delta}(\bar x) \\
		&  \geq \frac{1}{2}t_k^2 \left( V^{''}_{\delta}(\bar x;u) + {\rm d}^2 \delta_X(\bar x;-v^u)(u) \right)+o(t_k^2) > o(t_k^2).
	\end{aligned}
\end{equation*}
Hence we have a contradiction to \eqref{contradiction3} and consequently (\ref{second-order-growth3}) holds.
\end{proof}

Sufficient conditions for twice directional differentiability of the value function were discussed in the literature; cf.\ \cite{Auslender90,Bondarevsky16,Van2017}. We will discuss the details in the next section.

\section{Set constrained Systems}\label{setconstrained}
In this section we consider the minimax problem:
\begin{equation}\label{set-constrained}
  \min_{x\in X}\max_{y\in  Y(x)} f(x,y),
  \end{equation}where  the constraints are defined by the following set-constrained systems: \begin{eqnarray*}
X:= \{ x\in \mathbb{R}^n | \phi(x)\in C\},\quad
Y(x):= \{ y\in \mathbb{R}^m | \varphi (x,y) \in D\},
\end{eqnarray*}
where  $\phi:\mathbb{R}^n \rightarrow \mathbb{R}^{p}$ and $\varphi:\mathbb{R}^n \times \mathbb{R}^m  \to \mathbb{R}^q$,  $C\subseteq \mathbb{R}^{p}$ and $D\subseteq \mathbb{R}^{q}$ are closed and convex.

Denote the Lagrangian function of the minimax problem \eqref{set-constrained} by
\begin{eqnarray*}
L(x,y,\alpha,\beta) := f(x,y) + \phi(x)^T\alpha - \varphi(x,y)^T\beta,
\end{eqnarray*}
and the Lagrangian function of the maximization problem $\max_{y\in Y(\bar x)} f(\bar x,y)$ by 
\begin{eqnarray*}
L_{\max}(x,y,\beta) := f(x,y) - \varphi(x,y)^T\beta.
\end{eqnarray*}

Define the set of multipliers for the maximization problem $\max_{y\in Y(\bar x)} f(\bar x,y)$ by
$$\Lambda_{\max}(\bar y;\bar x):=\{\beta \in N_D(\varphi(\bar x,\bar y)) |  \nabla_y f(\bar x,\bar y)-\nabla_y \varphi(\bar x,\bar y)^T\beta=0 \}.$$

Define the projection  of the linearization cone to ${\rm gph}  Y$ at $(\bar x,\bar y)$ to space $\mathbb{R}^m$ in direction $u$ as   $$\mathbb{L}(\bar x,\bar y;u) :=\{h'\in \mathbb{R}^m|\nabla \varphi(\bar x,\bar y)(u,h')\in T_D(\varphi(\bar x,\bar y))\}.$$
By \cite[Proposition 4.2]{ABM21}, when MSCQ holds for system $\varphi(\bar x,y)\in D$ at $\bar y$, we have
$$\mathbb{L}(\bar x,\bar y;u) = D Y(\bar x,\bar y)(u).$$

\begin{lemma}\label{lemma4.5.1}
Suppose that $f$ and $\varphi$ are smooth and $Y$ satisfies the Robinson stability (RS) property at $(\bar x,\bar y)$. 
Suppose that  the regularity condition
\begin{equation}\label{regular29}
0=\nabla_y \varphi(\bar{x},\bar y)^T \beta, \beta \in N_D(\varphi(\bar{x},\bar y)) \cap\{\nabla \varphi(\bar{x},\bar y)(u, h)\}^{\perp} \Longrightarrow \beta=0
\end{equation}
holds for any $u \in T_{ X}(\bar x), h \in \mathbb{L}(\bar x,\bar y;u)$. Then the strong duality 
\begin{equation}\label{strongd}
\max\limits_{h \in \mathbb{L}(\bar x,\bar y;u)} \nabla f(\bar x,\bar y)^T(u,h) = \min\limits_{\beta \in \Lambda_{\max}(\bar y; \bar x)} \nabla_x f(\bar x,\bar y)^Tu - (\nabla_x \varphi(\bar x,\bar y)^T\beta) ^Tu
\end{equation} holds and the maximum with respect to $h$ can be attained.
\end{lemma}
\begin{proof}
Given any $u\in T_X(\bar x)$. We now show that the regularity condition \eqref{regular29} is equivalent to Robinson's CQ for the conic linear system $\nabla \varphi(\bar{x}, \bar y)( u, h) \in T_D(\varphi(\bar x,\bar y))$ holds at each $ h \in \mathbb{L}(\bar x,\bar y;u)$, i.e.,
$$
0=\nabla_y \varphi(\bar{x}, \bar y)^T \beta, \beta \in N_{T_D(\varphi(\bar{x}, \bar y))}(\nabla \varphi(\bar{x}, \bar y)( u, h)) \Longrightarrow \beta=0 .
$$
Recall that for any closed convex cone $\mathcal{K}$ and any $d \in \mathcal{K}, N_{\mathcal{K}}(d)=\mathcal{K}^o \cap\{d\}^{\perp}$; see e.g. \cite[Corollary 23.5.4]{Ro1970}. Since $D$ is convex, the tangent cone $T_D(\varphi(\bar{x}, \bar y))$ is closed and convex.
$$
\begin{aligned}
N_{T_D(\varphi(\bar{x}, \bar y))}(\nabla \varphi(\bar{x}, \bar y)( u, h)) & =\left\{T_D(\varphi(\bar{x}, \bar y))\right\}^{\circ} \cap\{\nabla \varphi(\bar{x}, \bar y)(u, h)\}^{\perp} \\
& =N_D(\varphi(\bar{x}, \bar y)) \cap\{\nabla \varphi(\bar{x}, \bar y)(  u, h)\}^{\perp},
\end{aligned}
$$
where the second equality follows from tangent-normal polarity in Proposition \ref{prop_polar}. 

Then, by \cite[Theorem 2.1]{BaiYe2023}, the strong duality \eqref{strongd} holds 
and the maximum with respect to $h$ can be attained (since the maximum is the directional derivative of the value function, see Lemma \ref{prop4.1} and Proposition \ref{prop2.2}).
\end{proof}

Note that the regularity condition \eqref{regular29} is equivalent to the first-order sufficient condition for metric subregularity (FOSCMS) in the direction $\varphi(\bar{x},\bar y)(u,h)$ when $D$ is convex \cite[Remark 4.1]{BaiYe2023}. Moreover, for inequalities and equalities systems, the FOSCMS is implied by the MFCQ.

Now, we can give first-order optimality conditions for the minimax problem  (\ref{set-constrained}).
  \begin{theorem}[First-order necessary conditions for set-constrained systems]\label{set-first}Let $(\bar x, \bar y)$ be a calm local minimax point to the minimax problem \eqref{set-constrained} where $f,\phi, \varphi$ are smooth. Suppose that the MSCQ holds for the system $\phi(x)\in C$ at $\bar x$, $\varphi(\bar x,y)\in D$ at $\bar y$, respectively. Suppose that $Y$ satisfies the Robinson stability (RS) property at $(\bar x,\bar y)$.  Suppose that the regularity condition \eqref{regular29} holds for any $u \in T_{ X}(\bar x), h \in \mathbb{L}(\bar x,\bar y;u)$. Then there exist $\alpha \in N_C(\phi(\bar x)), \beta \in N_D(\varphi(\bar x,\bar y))$,  such that
\begin{equation*}
\nabla_{(x,y)} L(\bar x,\bar y,\alpha,\beta) =0.
\end{equation*}
   \end{theorem}

\begin{proof}
 Using the MSCQ, the RS, and the convexity of $C$ and $D$, we get from \cite[Proposition 4.2]{ABM21} that
$$\widehat N_X(\bar x) = N_X(\bar x) = \nabla \phi(\bar x)^T N_C(\phi(\bar x)),$$
$$\widehat N_{Y(\bar x)}(\bar y) = N_{Y(\bar x)}(\bar y) = \nabla_y \varphi(\bar x,\bar y)^T N_D(\varphi(\bar x,\bar y)),$$
\begin{equation}\label{normaly}\widehat{N}_{\mathrm{gph} Y}(\bar{x}, \bar{y}) = N_{\mathrm{gph} Y}(\bar{x}, \bar{y}) = \nabla \varphi(\bar x,\bar y)^T N_D(\varphi(\bar x,\bar y)).\end{equation}

By Theorem \ref{1st-non1} (a), 
\begin{eqnarray*}
&& \sup_{h \in D Y(\bar x,\bar y)(u)} \nabla f(\bar{x}, \bar{y})^T(u,h) \geq 0 \quad \forall\, u \in T_X(\bar{x}),
\end{eqnarray*}
\begin{equation*}
\nabla_y f(\bar{x}, \bar{y})^Th \leq 0\quad \forall\, h \in T_{Y(\bar x)}(\bar{y}).
\end{equation*}
By \eqref{normalcone-equ}, $0 \in -\nabla_y f(\bar{x}, \bar{y}) + \widehat N_{Y(\bar x)}(\bar y)$. Thus, the multiplier set $\Lambda_{\max}(\bar y;\bar x)$ is nonempty.

By the strong duality \eqref{strongd}, we have 
$$\min\limits_{\beta \in \Lambda_{\max}(\bar y; \bar x)} \nabla_x f(x,y)^Tu - (\nabla_x \varphi(x,y)^T\beta) ^Tu \geq 0 \quad \forall\, u \in T_X(\bar{x}).$$
Thus, by \eqref{normalcone-equ}, for any $\beta \in \Lambda_{\max}(\bar y;\bar x)$, we have
$0 \in \nabla_x f(x,y) - \nabla_x \varphi(x,y)^T\beta + \widehat N_{X}(\bar x)$.
Then, we have the desired results.
\end{proof}

Next, we consider the special case where the constraint sets $X$ and $Y(x)$ involving only equalities and inequalities. Let $C=\mathbb{R}^{p_1}_- \times \{0\}^{p_2}$ and $D=\mathbb{R}^{q_1}_- \times \{0\}^{q_2}$. In this case,
the set
$$\mathbb{L}(\bar x,\bar y;u) :=\{h'|\nabla \varphi_i(\bar x,\bar y)^T (u,h') \leq 0, i \in I_{\varphi}(\bar x,\bar y), \nabla \varphi_j(\bar x,\bar y)^T (u,h') = 0, j=1,...,q_2\},$$
where  $I_\varphi(\bar x,\bar y):=\{i=1,...,q_1\mid\varphi_i(\bar x,\bar y)=0\}$.
If MSCQ holds for $\varphi(\bar x,y)\in D$ at $\bar y$, then the critical cone for the maximization problem $\max_{y\in Y(\bar x)} f(\bar x,y)$ becomes
 $$C_{\max}(\bar y;\bar x) = \mathbb{L}(\bar x,\bar y;u) 
	\cap \{ h' \in \mathbb{R}^m \mid \nabla_y f(\bar x,\bar y)^Th'\geq  0 \}.$$
	And when $\bar y$ is an optimal solution, 
	$$C_{\max}(\bar y;\bar x) = \mathbb{L}(\bar x,\bar y;u) 
	\cap \{ h' \in \mathbb{R}^m \mid \nabla_y f(\bar x,\bar y)^T h'=  0 \}.$$
If MSCQ holds for $\phi(x)\in C$ at $\bar x$, then the critical cone for the minimization problem $\min_{x\in X} V_{\tau}(x)$ becomes
\begin{eqnarray*}
	C_{\min}(\bar x,\bar y)
	&=& \mathbb{L}_X(\bar x)\cap \{ u|\sup\limits_{h \in \mathbb{L}(\bar x,\bar y;u)} \nabla f(\bar x,\bar y)^T(u,h) \leq  0\},
\end{eqnarray*} where 
$$\mathbb{L}_X(\bar x)=\left\{ u \in \mathbb{R}^n \;\middle |\;\begin{array}{l}
		\nabla \phi_i(\bar x)^T u \leq 0, \; i \in I_\phi(\bar x),\\
		\nabla \phi_j(\bar x)^T u = 0, \; j=1,\dots,p_2,\end{array}  \right \},$$
		where the index set $I_\phi(\bar x):=\{i=1,...,p_1 \mid \phi_i(\bar x)=0\}$, and when $\bar x$ is an optimal solution,
$$ C_{\min}(\bar x,\bar y)
	= \mathbb{L}_X(\bar x)\cap \{ u|\sup\limits_{h \in \mathbb{L}(\bar x,\bar y;u)} \nabla f(\bar x,\bar y)^T(u,h) =  0\}.$$

If MSCQ holds for $\varphi(\bar x,y)\in D$ at $\bar y$, then the set $C(\bar x,\bar y;u)$ becomes
$$
	C(\bar x,\bar y;u)= \mathbb{L}(\bar x,\bar y;u) 
	\cap \{ h' \in \mathbb{R}^m \mid \nabla f(\bar x,\bar y)^T (u,h') = 0 \}.
$$

The set of multipliers for the minimax problem \eqref{set-constrained} becomes
\begin{eqnarray*}
 \Lambda(\bar x,\bar y) &=&  \{ (\alpha,\beta):=((\alpha_1,\alpha_2),(\beta_1,\beta_2)) \in (\mathbb{R}^{p_1}_+ \times \mathbb{R}^{p_2}) \times (\mathbb{R}^{q_1}_+ \times \mathbb{R}^{q_2}) \mid   \\
& & \quad \nabla f(\bar x,\bar y)+(\nabla \phi(\bar x)^T\alpha, 0)- \nabla \varphi(\bar x,\bar y)^T\beta=0, \alpha_1 \perp \phi_{\leq}(\bar x) ,\beta_1 \perp \varphi_{\leq}(\bar x,\bar y)  \}.
\end{eqnarray*}
Moreover, the set of multipliers
\begin{eqnarray*}
\Lambda_{\max}(\bar y; \bar x)&=& \left \{ \beta:=(\beta_1,\beta_2) \in \mathbb{R}^{q_1}_+ \times \mathbb{R}^{q_2} \mid \nabla_y f(\bar x,\bar y)- \nabla_y \varphi(\bar x,\bar y)^T\beta=0,  \beta_1 \perp \varphi_{\leq}(\bar x,\bar y) \right \},
\end{eqnarray*}
where $\phi_{\leq}(\bar x):=(\phi_1(\bar x), ... ,\phi_{p_1}(\bar x))^T$ and $\varphi_{\leq}(\bar x,\bar y):=(\varphi_1(\bar x,\bar y), ... ,\varphi_{q_1}(\bar x,\bar y))^T$.

	When reducing to inequalities systems, by Proposition \ref{prop2.3}, Theorem \ref{set-first} becomes the following one.

\begin{corollary}[inequalities and equalities systems (first-order)]\label{coro5.1}
Let $(\bar x, \bar y)$ be a calm local minimax point to the minimax problem \eqref{set-constrained} with  $C=\mathbb{R}^{p_1}_- \times \{0\}^{p_2}$ and $D=\mathbb{R}^{q_1}_- \times \{0\}^{q_2}$ where $f,\phi, \varphi$ are smooth. Suppose that  the MSCQ holds for the system $\phi(x)\in C$ at $\bar x$ (e.g., $\phi$ is affine or MFCQ holds at $\bar x$) and one of the following assumptions hold:  (i) MFCQ holds at $(\bar x,\bar y)$;  (ii) RCRCQ holds  at $(\bar x,\bar y)$; (iii) $\varphi(x, y)=a(x) + By + c$, where $a: \mathbb{R}^n \to \mathbb{R}^p$ is continuous, $B \in \mathbb{R}^{p \times m}$ and $c \in \mathbb{R}^p$ and $Y(x)$ is nonempty near $\bar x$. Then there exist 
$\alpha:=(\alpha_1,\alpha_2) \in \mathbb{R}^{p_1}_+ \times \mathbb{R}^{p_2},\beta:=(\beta_1,\beta_2) \in \mathbb{R}^{q_1}_+ \times \mathbb{R}^{q_2}$,  such that $\alpha_1 \perp \phi_{\leq}(\bar x), \beta_1 \perp \varphi_{\leq}(\bar x,\bar y)$ and 
\begin{equation*}
\nabla_{(x,y)} L(\bar x,\bar y,\alpha,\beta) =0.
\end{equation*} 
\end{corollary}

{ Compared to \cite[Theorems 3.1 and 3.3]{DaiZh-2020}, we have derived the same optimality condition under much weaker constraint qualification. } In particular, we do not need to assume the Jacobian uniqueness condition.

Now, to give second-order sufficient optimality conditions for inequalities and equalities systems, by  Corollary \ref{sufficient2}, we need to study the second-order directional differentiability of the localized value function $V_\delta$.

We say that the strong second-order sufficient condition in direction $u$ (SSOSC$_u$) holds at $(\bar x,\bar y)$ if $\Lambda_{\max}(\bar x,\bar y) \neq \emptyset$ and
$$\inf_{\beta \in \Lambda^2_{\operatorname{max}}(\bar x,\bar y,u)} h^T \nabla^2_{yy} L_{\max}(\bar y,\beta;\bar x)h <0 \quad \forall h \in C_{\max}(\bar y;\bar x) \setminus \{0\},$$
where 
$$\Lambda^2_{\operatorname{\max}}(\bar x,\bar y,u):=\argmin_{\beta \in \Lambda_{\max} (\bar y;\bar x)} \left \{ \nabla_x f(\bar x,\bar y)^T u-(\nabla_x \varphi (\bar x,\bar y)^T \beta)^T u\right \}$$ is the solution set of the right-hand side of \eqref{strongd} and
$C_{\max}(\bar y;\bar x)$ is the critical cone for the maximization problem $\max_{y\in Y(\bar x)} f(\bar x,y)$.
When SSOSC$_u$  holds at $(\bar x, \bar y)$, $\bar y$ is a strict local maximizer and hence for $\delta>0$ small enough, the solution set is a singleton $S_\delta(\bar x)=\{\bar y\}$.

\begin{proposition}\label{V2diff1}
          {  
    Assume that  $f,\phi, \varphi$ are twice continuously differentiable. Suppose that  MFCQ  holds for the system $\varphi( \bar x,y)\in D$ at  $ \bar y$ and  SSOSC$_u$ holds at $(\bar x,\bar y)$ for any $$u \in \mathbb{L}_X(\bar x) \cap \{u|\max_{h'\in \mathbb{L}(\bar x,\bar y;u)}{\rm d}f(\bar x,\bar y)(u,h')=0\} \setminus \{0\}.$$  Let $\delta>0$ be small enough such that $\bar y$ is the only maximizer in $S_\delta(\bar x)$. 
	 Then, 	$V_{\delta}$ is Lipschitz continuous and 
	 twice directional differentiable} at $\bar x$ and  for   any $u \in C_{\min}(\bar x,\bar y) \setminus \{0\}$,
	$$V^{''}_{\delta}(\bar x;u) \geq \max_{h \in \mathbb{L}^2(\bar x,\bar y;u)}\min_{\beta \in \Lambda^2_{\operatorname{max}}(\bar x,\bar y,u)} \nabla^2_{(x,y)} L_{\operatorname{max}}(\bar x,\bar y,\beta)((u,h),(u,h)).$$
\end{proposition}
\begin{proof}  By definition, given $\bar x$, $V_\delta(\bar x)$ is the maximum value of the following parametric program:
$$\max_y f(\bar x,y) \qquad s.t.\ \  \varphi(\bar x,y)\in D, \|y-\bar y\|\leq \delta.$$
To apply the sensitivity results, we reformulate the nonsmooth inequality constraint  $\|y-\bar y\|\leq \delta$ as a smooth inequality constraint $\|y-\bar y\|^2\leq \delta^2$. Since  the inequality constraint $\|y-\bar y\|^2\leq \delta^2$ is inactive at $\bar y$,     MFCQ holds for the system $\varphi(\bar x,y)\in D, \|y-\bar y\|^2\leq \delta^2 $ at $\bar  y$ if and only if it  holds for the system $\varphi(\bar x,y)\in D$ at $\bar y$. Since the feasible region is uniformly bounded and MFCQ holds at $\bar y$ which is the only solution in $S_\delta(\bar x)$  by  \cite[Theorem 5.1]{Gauvin1982}, $V_\delta$ is Lipschitz continuous at $\bar x$.  By \cite[Theorem 4.2]{Bondarevsky16}, $V_\delta(x)$ is  twice directional differentiable at $\bar x$ and the lower bound follows.
\end{proof}

Now, we can give second-order optimality conditions for the minimax problem  (\ref{set-constrained}) with inequalities and equalities systems.

 \begin{theorem}[inequalities and equalities systems (second-order)]\label{second-dual}Assume that  $f,\phi, \varphi$ are twice continuously differentiable.         Let $(\bar x,\bar y) \in X \times  Y(\bar x)$, $C=\mathbb{R}^{p_1}_- \times \{0\}^{p_2}$ and $D=\mathbb{R}^{q_1}_- \times \{0\}^{q_2}$. Suppose that the MSCQ holds for the system $\phi(x)\in C$ at $\bar x$, and that MFCQ holds for the system $\varphi(\bar x,y)\in D$ at $\bar y$.
  \begin{itemize}
 \item[(a)] Suppose that  $(\bar x,\bar y)$ is a calm local minimax point to problem \eqref{set-constrained}. Then the following second-order necessary optimality conditions for the maximization hold:
 for any $h \in C_{\max}(\bar y;\bar x)$, there exists a multiplier $\beta \in \Lambda_{\max}(\bar y; \bar x)$ such that
\begin{equation*}
h^T \nabla^2_{yy} L_{\max}(\bar y,\beta;\bar x) h  \leq 0,
\end{equation*}
and for any $u \in C_{\min}(\bar x,\bar y)$, there exist $h \in C(\bar x,\bar y;u)$ { such that for any multiplier $\beta \in \Lambda^2_{\operatorname{max}}(\bar x,\bar y;u)$,  there exists $\alpha$ such that $(\alpha, \beta)\in \Lambda(\bar x,\bar y)$ and}
\begin{equation}\label{withh}
\nabla^2_{(x,y)} L(\bar x,\bar y,\alpha,\beta)((u,h),(u,h))  \geq 0.
\end{equation}

  \item[(b)] 
For any 
$ u\in \mathbb{L}_X(\bar x)\cap \{ u|\sup\limits_{h \in \mathbb{L}(\bar x,\bar y;u)} \nabla f(\bar x,\bar y)^T(u,h) =  0\} \setminus \{0\}, $  suppose that SSOSC$_u$ holds at $(\bar x,\bar y)$ and  there exists $h \in C(\bar x,\bar y;u)$ such that for any multiplier $\beta \in \Lambda^2_{\operatorname{max}}(\bar x,\bar y;u)$,  there exists $\alpha$ such that $(\alpha, \beta)\in \Lambda(\bar x,\bar y)$ and
\begin{equation}\label{withhsuf}
\nabla^2_{(x,y)} L(\bar x,\bar y,\alpha,\beta)((u,h),(u,h)) > 0.
\end{equation} 
Then $(\bar x,\bar y)$ is a calm local minimax point to problem \eqref{set-constrained} {with the second-order growth condition.
}
\end{itemize}
   \end{theorem}

\begin{proof}

 (i) Since $f$ is twice continuously differentiable, it is obvious that $f$ is twice semidifferentiable. The mapping $Y$ is semidifferentiable at $\bar x$ for $\bar y$ and calm around $(\bar x,\bar y)$ by Proposition \ref{prop2.3}.

 (ii) We show that under the MSCQ and MFCQ assumptions, the nonemptyness of the multiplier set $\Lambda(\bar x,\bar y)$ is equivalent to the first-order optimality conditions \eqref{firstordercond11} and \eqref{firstordercond21}.


Suppose that $\Lambda(\bar x,\bar y)$ is nonempty. For any $(\alpha,\beta) \in \Lambda(\bar x,\bar y)$, with \cite[Proposition 4.2]{ABM21}, we have  
\begin{equation}\label{xnormal}
	-\nabla_x f(\bar x,\bar y)+\nabla_x \varphi(\bar x,\bar y)^T\beta=\nabla_x \phi(\bar x)^T\alpha \in N_{X}(\bar x).
	\end{equation}
By \eqref{normalcone-equ}, this is equivalent to 
\begin{align} \langle -\nabla_x f(\bar x,\bar y)+\nabla_x \varphi(\bar x,\bar y)^T\beta,u \rangle \leq 0 \quad  \text { for all } u \in T_{ X}(\bar x). \label{u2}
   \end{align}

By the linear programming duality theorem,
\begin{equation}\label{strongdua} \max\limits_{h \in \mathbb{L}(\bar x,\bar y;u)} \nabla f(\bar x,\bar y)^T(u,h) = \min\limits_{\beta \in \Lambda_{\max}(\bar y; \bar x)} \nabla_x f(x,y)^Tu - (\nabla_x \varphi(x,y)^T\beta) ^Tu.
	\end{equation}
and the maximum with respect to $h$ can be attained.

Thus, for each $u \in \mathbb{R}^n$, there exist $\tilde h \in \mathbb{L}(\bar x,\bar y;u)$ and $\tilde \beta \in \Lambda_{\max}(\bar y; \bar x)$ such that 
\begin{equation}\label{=}
-\nabla_y f(\bar x,\bar y)^T\tilde h =  (\nabla_x \varphi(\bar x,\bar y)^T\tilde \beta)^Tu.
\end{equation}
Since \eqref{u2} holds for any $\beta \in \Lambda_{\max}(\bar y; \bar x)$, we can plug $\tilde \beta$ into \eqref{u2} and then use the fact in \eqref{=}. Then, we have \eqref{firstordercond11}. Similarly, we can obtain \eqref{firstordercond21} using $\beta \in \Lambda_{\max}(\bar y; \bar x)$.

We still need to show that \eqref{firstordercond11} together with \eqref{firstordercond21} imply $\Lambda(\bar x,\bar y) \neq \emptyset$. This follows directly from the proof of Theorem \ref{set-first}.

(iii) We show that for any $u\in C_{\min}(\bar x; \bar y)$,  there exist $-v^u\in N_X(\bar x)\cap \{u\}^\perp, (\xi_1^u,\xi_2^u) \in N_{\operatorname{gph}Y}(\bar x,\bar y)$ such that $\langle (v^u,0)+(\xi_1^u,\xi_2^u), (u',h') \rangle =\nabla  f(\bar x,\bar y)^T(u',h')$ for any $(u',h') \in \mathbb{R}^n \times \mathbb{R}^m$,  and that the value ${\rm d}^2 \delta_X(\bar x;-v^u)(u)$ is finite and ${\rm d}^2 \delta_{\operatorname{gph}Y}((\bar x,\bar y);\nabla  f(\bar x,\bar y))(u,h)$ is finite for any $h \in C(\bar x,\bar y,u)$.

For each $u\in C_{\min}(\bar x; \bar y)$, let $\tilde{\beta}$ be a solution to the right-hand side of \eqref{strongdua}, $\tilde h$  be a solution to the left-hand side of \eqref{strongdua}, and $(\tilde \alpha, \tilde \beta )\in \Lambda(\bar x,\bar y)$. Then, 
\begin{equation}\label{1st-smooth}
	\nabla f(\bar x,\bar y) + (\nabla \phi (\bar x)^T \tilde \alpha,0) -\nabla \varphi(\bar x,\bar y)^T \tilde \beta =0,
\end{equation}
\begin{equation}\label{V1=0}
0=\nabla f(\bar x,\bar y)^T(u,\tilde h) =  \nabla_x f(x,y)^Tu -  (\nabla_x \varphi(\bar x,\bar y)^T\tilde \beta)^Tu,
\end{equation}
where \eqref{1st-smooth} holds by the definition of $\Lambda(\bar x,\bar y)$ and \eqref{V1=0} holds since $u\in C_{\min}(\bar x; \bar y)$.

Let $v^u:=-\nabla \phi(\bar{x})^T\tilde \alpha$, $(\xi_1^u,\xi_2^u):=(\nabla_x \varphi(\bar x,\bar y)^T\tilde \beta ,\nabla_y \varphi(\bar x,\bar y)^T\tilde \beta)$. Then,
$$-v^u \in \{u\}^{\perp}, \langle (\xi_1^u,\xi_2^u),(u,\tilde h) \rangle =0 \quad \mbox{ by } \eqref{1st-smooth},\eqref{V1=0},$$
$$-v^u=\nabla \phi(\bar{x})^T\tilde \alpha   \in N_{X}(\bar x)  \quad \mbox{ by } \eqref{xnormal},$$
$$(\xi_1^u,\xi_2^u) = \nabla \varphi(\bar x,\bar y)^T\tilde \beta \in N_{\operatorname{gph} Y}(\bar x,\bar y) \quad \mbox{ by } \eqref{normaly},$$
$$ \langle (v^u,0) +(\xi_1^u,\xi_2^u), (u',h') \rangle=\nabla f(\bar x,\bar y)^T(u',h') \ \forall u',h'    \quad \mbox{ by } \eqref{1st-smooth}.$$
Moreover, by \cite[Theorem 3.3]{ABM21}, the subderivatives are finite. 

(iv) Using \cite[Proposition 2.3, Proposition 2.5 (ii)(iii)]{Ma2023},
\begin{equation*}
\begin{aligned}
-{\rm d}^2 \delta_{Y(\bar x)}(\bar y;\nabla_yf(\bar x,\bar y))(h) & = & -\max_{\beta \in \Lambda_{\max}(\bar y; \bar x)} \{ \langle \beta, \nabla_{yy}^2 \varphi(\bar x,\bar y)(h,h) \rangle \} \\
& = &\min_{\beta \in \Lambda_{\max}(\bar y; \bar x)} \{- \langle \beta, \nabla_{yy}^2 \varphi(\bar x,\bar y)(h,h) \rangle \}.
\end{aligned}
\end{equation*}
For any $u$, we have, by \cite[Proposition 2.3, Proposition 2.7 (ii)(iii)]{Ma2023}, 
\begin{equation*}
{\rm d}^2 \delta_{X}(\bar x;\nabla \phi(\bar{x})^T\tilde \alpha)(u) = \langle \tilde \alpha, \nabla^2 \phi(\bar x)(u,u) \rangle ,
\end{equation*}
\begin{equation*}
-{\rm d}^2 \delta_{\operatorname{gph}Y}((\bar x,\bar y);(\nabla_x \varphi(\bar x,\bar y)^T\tilde \beta ,\nabla_y \varphi(\bar x,\bar y)^T\tilde \beta))(u,\tilde h) =-\langle \tilde \beta, \nabla^2 \varphi(\bar x,\bar y)(( u,\tilde h),(u,\tilde h)) \rangle .
\end{equation*}
Since $\tilde h \in \mathbb{L}(\bar x,\bar y;u)$ and $0=\nabla f(\bar x,\bar y)^T(u,\tilde h)$, we have $\tilde h \in C(\bar x,\bar y;u)$.

With all the above discussions and Theorem \ref{main1} (a), we have the necessary optimality conditions.

(v) 
 By Corollary \ref{sufficient2} and Proposition \ref{V2diff1}, we obtain the sufficient optimality condition.

\end{proof}

The proof for Theorem \ref{second-dual} shows that the second-order necessary condition with respect to $x$ can be equivalently stated as follows:
	for any $u \in C_{\min}(\bar x,\bar y)$, there exists $h \in C(\bar x,\bar y;u)$ such that, 
	for any multiplier $\beta \in \Lambda^2_{\operatorname{max}}(\bar x,\bar y;u)$, 
	there exists $\alpha$ with $(\alpha,\beta) \in \Lambda(\bar x,\bar y)$ and \eqref{withh} holds. 
	Thus, in Theorem \ref{second-dual}, we have established no-gap second-order sufficient and necessary optimality conditions.

We now compare our second-order optimality conditions with the one obtained by Dai and Zhang in \cite{DaiZh-2020}. The following result offers an intuitive characterization of the direction $h$ that arises in the optimality conditions.
\begin{corollary}   Let  $C=\mathbb{R}^{p_1}_- \times \{0\}^{p_2}$, $D=\mathbb{R}^{q_1}_- \times \{0\}^{q_2}$ and $(\bar x,\bar y) \in X \times  Y(\bar x)$.  Suppose that $\nabla_x f(\bar x,\bar y) \in \nabla \phi(\bar x)^T N_C(\phi(\bar x))$ and the MSCQ holds for the system $\phi(x)\in C$ at $\bar x$, that { MFCQ holds for the system $\varphi(\bar x,y)\in D$ at $\bar y$, and that for any $u \in C_{\min}(\bar x,\bar y)$, there exists a multiplier $\beta \in \Lambda^2_{\max}(\bar x,\bar y;u)$ such that $\nabla_{yy}^2 L_{\max}(\bar y,\beta;\bar x) \prec 0.$} Let $h^{*}$ be defined in \eqref{equa4.18}.
	\begin{itemize}
		\item[(i)] Let $\left(\bar x,\bar y\right)$ be a local minimax point of problem \eqref{set-constrained}. Then for any $u \in C_{\min}(\bar x,\bar y)$, there exists $\alpha$ such that $(\alpha, \beta)\in \Lambda(\bar x,\bar y)$ and
		\begin{equation}\label{withouthn}
			\nabla^2_{(x,y)} L(\bar x,\bar y,\alpha,\beta)((u,h^*),(u,h^*)) \geq 0.
		\end{equation}
		\item[(ii)] If for any $u \in C_{\min}(\bar x,\bar y) \setminus \{0\}$, there exists $\alpha$ such that $(\alpha, \beta)\in \Lambda(\bar x,\bar y)$ and
		\begin{equation}\label{withouths}
			\nabla^2_{(x,y)} L(\bar x,\bar y,\alpha,\beta)((u,h^*),(u,h^*)) > 0.
		\end{equation}
		Then, $(\bar x,\bar y)$ is a local minimax point to problem \eqref{set-constrained} with the second-order growth condition. 
	\end{itemize}
\end{corollary}
\begin{proof}
By \cite[Lemma 5.1]{Ma2023}, a local minimax point is equivalent to a calm local minimax point under the negative definiteness assumption.  
{ Note that, under this assumption, SSOSC$_u$ holds at $(\bar x,\bar y)$ for any $u \in \mathbb{L}_X(\bar x) \cap \{u \mid \max_{h' \in \mathbb{L}(\bar x,\bar y;u)} {\rm d}f(\bar x,\bar y)(u,h') = 0\} \setminus \{0\}$.}

 Since $\nabla_{yy}^2 L(\bar x,\bar y,\alpha,\beta)=\nabla_{yy}^2 L_{\max}(\bar y,\beta;\bar x) \prec 0$ and  $C(\bar x,\bar y;u)$ is nonempty for any $u \in C_{\min}(\bar x,\bar y)$, the supremum
\begin{equation*}
	\max_{h\in C(\bar x,\bar y;u) }\nabla^2_{(x,y)} L(\bar x,\bar y,\alpha,\beta)((u,h),(u,h))
\end{equation*}
can be attained and the optimal point, by the KKT condition, is 
\begin{equation}\label{equa4.18}
	\begin{aligned}
		h^* &:= -\nabla_{yy}^2 L(\bar x,\bar y,\alpha,\beta)^{-1} 
		\Bigl[ \nabla_{xy}^2L(\bar x,\bar y,\alpha,\beta)^Tu  \\
		&\quad\quad\quad\quad\quad\quad\quad\quad\quad\quad\quad -\sum_{i=1}^{q_1}\lambda_i \nabla_y \varphi_i (\bar x,\bar y) 
		-\sum_{j=1}^{q_2} \eta_j  \nabla_y \varphi_j(\bar x,\bar y) 
		-\zeta \nabla_y f(\bar x,\bar y)\Bigr] \\
		&= -\nabla_{yy}^2 L(\bar x,\bar y,\alpha,\beta)^{-1} 
		\Bigl[\nabla_{xy}^2L(\bar x,\bar y,\alpha,\beta)^Tu \\
		&\quad\quad\quad\quad\quad\quad\quad\quad\quad\quad\quad -\sum_{i=1}^{q_1}(\lambda_i+\zeta \beta_i)\nabla_y \varphi_i (\bar x,\bar y) 
		-\sum_{j=1}^{q_2}(\eta_j +\zeta \beta_j)\nabla_y \varphi_j(\bar x,\bar y)\Bigr] \\
		&= -\nabla_{yy}^2 L(\bar x,\bar y,\alpha,\beta)^{-1}\nabla_{xy}^2L(\bar x,\bar y,\alpha,\beta)^Tu + \nabla_{yy}^2 L(\bar x,\bar y,\alpha,\beta)^{-1} \Xi \nabla_y \varphi (\bar x,\bar y),
	\end{aligned}
\end{equation}
where $\lambda \in \mathbb{R}^{q_1}_+, \eta \in \mathbb{R}^{q_2}, \zeta\in \mathbb{R}, \lambda_i \nabla \varphi_i(\bar x,\bar y) (u,h^*)=0,i=1,...,q_1$, and
$$\Xi:=(\lambda+\zeta \beta_1,\eta +\zeta \beta_2)\in \mathbb{R}^{q_1+q_2}.$$
Note that we used the fact $\nabla_y f(\bar x,\bar y)- \nabla_y \varphi(\bar x,\bar y)^T\beta=0$ to replace the term $\nabla_y f(\bar x,\bar y)$ in $h^*$.

\end{proof}

Plug \eqref{equa4.18} into $\nabla^2_{(x,y)} L(\bar x,\bar y,\alpha,\beta)((u,h^*),(u,h^*))$, we get
\begin{eqnarray*}
\begin{aligned}
&\nabla^2_{(x,y)} L(\bar x,\bar y,\alpha,\beta)((u,h^*),(u,h^*))\\
& =u^T [\nabla_{xx}^2 L -\nabla_{xy}^2 L(\nabla_{yy}^2 L)^{-1}\nabla_{yx}^2 L](\bar x,\bar y,\alpha,\beta) u \\
& - \Xi\nabla_y \varphi (\bar x,\bar y)^T\nabla_{yy}^2 L(\bar x,\bar y,\alpha,\beta)^{-1}\nabla_{yx}^2 L(\bar x,\bar y,\alpha,\beta)u\\
&+ \left[\Xi\nabla_y \varphi (\bar x,\bar y)^T + u^T \nabla_{xy}^2 L(\bar x,\bar y,\alpha,\beta) \right] \nabla_{yy}^2 L(\bar x,\bar y,\alpha,\beta)^{-1} \Xi \nabla_y \varphi (\bar x,\bar y) .
\end{aligned}
\end{eqnarray*}

Using the above representation, the second-order optimality conditions in \eqref{withh} and \eqref{withhsuf} can be reformulated without involving \(h\); see \eqref{withouthn} and \eqref{withouths}. This formulation is more explicit than that in \cite[Theorem 3.1, equation (3.7)  and Theorem 3.2]{DaiZh-2020}, as it avoids computing the inverse of the block matrix \(K(x)\) appearing therein.

{ In \cite[Corollary 5.1]{Ma2023}, for the decoupled constraint case, an upper estimate of   $\nabla^2_{(x,y)} L(\bar x,\bar y,\alpha,\beta)((u,h^*),(u,h^*))$ is obtained by relaxing the feasibility of \(h\) to the whole space. Since this only provides an upper bound, explicit second-order sufficient conditions without involving \(h\) are not derived there. In contrast, our result above provides an exact and explicit form of both the necessary and sufficient optimality conditions.}

\section{Concluding Remarks}\label{concl}
In this paper, we extend the concept of calm local minimax points to coupled-constrained minimax problems. We derive both first- and second-order necessary and sufficient optimality conditions for such points in the nonsmooth, nonconvex–nonconcave setting and in the smooth setting.


\begin{thebibliography}{99}

\bibitem{Aleksander18} Aleksander M., Aleksandar M., Ludwig S., Dimitris T., and Adrian V., 2018. Towards deep learning models resistant to adversarial attacks. International Conference on Learning Representations.

\bibitem{Aubin84} Aubin J.P., 1984. Lipschitz behavior of solutions to convex minimization problems. Mathematics of Operations Research, 9(1), pp. 87-111.


\bibitem{Auslender90} Auslender, A. and Cominetti, R., 1990. First and second order sensitivity analysis of nonlinear programs under directional constraint qualifications. Optimization, 21, pp. 351-363.




\bibitem{BaiYe2023} Bai K. and Ye J.J., 2024. Directional derivative of the value function for parametric set-constrained optimization problems. Journal of Optimization Theory and Applications, 203(2), pp.1355-1384.

\bibitem{Bednarczuk2020} Bednarczuk E.M., Minchenko L.I. and Rutkowski K.E., 2020. On Lipschitz-like continuity of a class of set-valued mappings. Optimization, 69(12), pp. 2535-2549.

\bibitem{Benko-2021} Benko, M., 2021. On inner calmness*, generalized calculus, and derivatives of the normal cone mapping. Journal of Nonsmooth Analysis and Optimization, 2, pp. 1-27.

\bibitem{BGO19} Benko, M., Gfrerer, H. and Outrata, J.V., 2019. Calculus for directional limiting normal cones and subdifferentials. Set-Valued and Variational Analysis, 27, pp. 713-745.

\bibitem{Bondarevsky16} Bondarevsky, V., Leschov, A. and Minchenko, L., 2016. Value functions and their directional derivatives in parametric nonlinear programming. Journal of Optimization Theory and Applications, 171(2), pp. 440-464.

\bibitem{BonSh00} Bonnans J.F. and Shapiro A., 2000. Perturbation Analysis of Optimization Problems. Springer, New York.


\bibitem{DaiWa2022}  Dai Y.H., Wang J. and Zhang L., 2024. Optimality conditions and numerical algorithms for a class of linearly constrained minimax optimization problems. SIAM Journal on Optimization, 34(3), pp. 2883-2916.

\bibitem{DaiZh-2020}  Dai Y.H. and Zhang L.W., 2020. Optimality conditions for constrained minimax optimization. CSIAM Transactions on Applied Mathematics. 1, pp. 296-315.

\bibitem{DaiZh-2022}  Dai Y.H. and Zhang L.W., 2022. The Rate of Convergence of Augmented Lagrangian Method for Minimax Optimization Problems with Equality Constraints. Journal of the Operations Research Society of China, pp. 1-33.

\bibitem{Dontchev2006}  Dontchev A.L., Quincampoix, M. and Zlateva, N., 2006. Aubin criterion for metric regularity. Journal of Convex Analysis, 13(2), p. 281.

\bibitem{DoRo2009} Dontchev A.L. and Rockafellar, R.T., 2009. Implicit functions and solution mappings: A view from variational analysis (Vol. 616). New York: Springer.



\bibitem{Gauthier2019} Gauthier G., Hugo B., Gaetan V., Pascal V., and Simon L.J., 2019. A variational inequality perspective on generative adversarial networks. International Conference on Learning Representations.

\bibitem{Gauvin1982} Gauvin J. and Dubeau F., 1982. Differential properties of the marginal function in mathematical programming, Mathematical Programming Study, 19, pp. 101-119.

\bibitem{Gfrerer2017} Gfrerer H, Mordukhovich BS, 2017. Robinson stability of parametric constraint systems via variational analysis. SIAM J. Optim. 27(1), pp.
438–465.

\bibitem{Goktas2022} Goktas D. and Greenwald A., 2022. Gradient Descent Ascent in Min-Max Stackelberg Games. arXiv preprint arXiv:2208.09690.

\bibitem{Goktas2023} Goktas D. and Greenwald A., 2021. Convex-concave min-max stackelberg games. Advances in Neural Information Processing Systems, 34, pp. 2991-3003.

\bibitem{GoktasZhao2022} Goktas D., Zhao J. and Greenwald A., 2022. Robust no-regret learning in min-max Stackelberg games. arXiv preprint arXiv:2203.14126.



\bibitem{Ian2014}  Ian G., Jean P.A., Mehdi M., Bing X., David W.F., Sherjil O., Aaron C., and Yoshua B., 2014. Generative adversarial nets. Advances in Neural Information Processing Systems.

\bibitem{JC22} Jiang J. and Chen X., 2023. Optimality conditions for nonsmooth nonconvex-nonconcave min-max problems and generative adversarial networks. SIAM Journal on Mathematics of Data Science, 5(3), pp. 693-722.

\bibitem{JNJ20} Jin C., Netrapalli P. and Jordan M., 2020. What is local optimality in nonconvex-nonconcave minimax optimization? International Conference on Machine Learning.


\bibitem{Xu2023} Liu X., Xu M. and Zhang L., 2023. Second-order optimality conditions for bi-local solutions of bilevel programs. Optimization, pp.1-23.

\bibitem{LuMei2023} Lu Z. and Mei S., 2024. A first-order augmented Lagrangian method for constrained minimax optimization. Mathematical Programming, pp.1-42.

\bibitem{Ma2023} Ma X., Yao W., Ye J.J. and Zhang J., 2025. Calm local optimality for nonconvex-nonconcave minimax problems. Set-Valued and Variational Analysis, 33(2), p.13.

\bibitem{Mehlitz2022} Mehlitz P. and Minchenko L.I., 2022. R-regularity of set-valued mappings under the relaxed constant positive linear dependence constraint qualification with applications to parametric and bilevel optimization. Set-Valued and Variational Analysis, 30(1), pp.179-205.

\bibitem{Minchenko2011} Minchenko, L. and Stakhovski, S., 2011. Parametric nonlinear programming problems under the relaxed constant rank condition. SIAM Journal on Optimization, 21(1), pp. 314-332.


\bibitem{M06}  Mordukhovich B.S., 2006. {Variational Analysis and Generalized Differentiation I: Basic Theory, vol. 330}, Springer Science \& Business Media, New York.

\bibitem{ABM21} Mohammadi A., Mordukhovich B. and Sarabi M., 2021. Parabolic regularity in geometric variational analysis. Transactions of the American Mathematical Society, 3, pp. 1711-1763.






\bibitem{Ro1970}Rockafellar, R.T., 1970. Convex Analysis, Princeton University Press, Princeton, New Jersey.

\bibitem{RoWe98} Rockafellar R.T and Wets R.J.-B., 1998. Variational Analysis, Springer, Berlin.

\bibitem{Still2018} Still G., 2018. Lectures on parametric optimization: An introduction. Optimization Online, p.2.

\bibitem{Tsaknakis2021} Tsaknakis I., Hong M. and Zhang S., 2023. Minimax problems with coupled linear constraints: computational complexity and duality. SIAM Journal on Optimization, 33(4), pp. 2675-2702.

\bibitem{Van2017} Van Nghi, T. and Tam, N.N., 2017. Continuity and directional differentiability of the value function in parametric quadratically constrained nonconvex quadratic programs. Acta Mathematica Vietnamica, 42(2), pp. 311-336.

\bibitem{YeYe1997} Ye J.J. and Ye X.Y., 1997. Necessary optimality conditions for optimization problems with variational inequality constraints. Mathematics of Operations Research, 22(4), pp. 977-997.


\bibitem{YeZhou18} Ye J.J. and Zhou J., 2018. Verifiable sufficient conditions for the error bound property of second-order cone complementarity problems. Mathematical Programming, 171, pp. 361-395.

\bibitem{Zhang2022} Zhang G., Poupart P. and Yu Y., 2022. Optimality and stability in non-convex smooth games. Journal of Machine Learning Research, 23, pp. 1-71.

\bibitem{ZhangWang2022} Zhang H., Wang J., Xu Z. and Dai Y.H., 2022. Primal dual alternating proximal gradient algorithms for nonsmooth nonconvex minimax problems with coupled linear constraints. arXiv preprint arXiv:2212.04672.
\end{thebibliography}
\end{document}